\documentclass[11pt,reqno]{amsart}
\usepackage{color}
\usepackage[english]{babel}
\usepackage{amsfonts}
\usepackage{amsmath}
\usepackage{amsthm}
\usepackage{amssymb}
\usepackage[misc]{ifsym}
\usepackage{bbm}
\usepackage{mathrsfs}
\usepackage{esint}
\usepackage{faktor}
\usepackage[utf8]{inputenc}
\usepackage{afterpage}
\usepackage[left=2.9cm,right=2.9cm,top=2.8cm,bottom=2.8cm]{geometry}
\usepackage{graphicx}
\usepackage{enumitem}
\usepackage[dvipsnames]{xcolor}
\usepackage[colorlinks=false,hidelinks,urlcolor=blue,citecolor=blue,linkcolor=blue,linktocpage,pdfpagelabels,bookmarksnumbered,bookmarksopen]{hyperref}

\newcommand{\N}{\mathbb{N}}

\newcommand{\R}{\mathbb{R}}

\newcommand{\C}{\mathcal{C}}
\newcommand{\D}{\mathcal{D}}

\newcommand{\dz}{\, {\rm d} z}
\newcommand{\dt}{\, {\rm d} t}

\newcommand{\dtau}{\, {\rm d} \tau}
\newcommand{\dxi}{\, {\rm d} \xi}
\newcommand{\dzeta}{\, {\rm d} \zeta}
\newcommand{\eps}{\varepsilon}

\newcommand{\beq}{\begin{equation}}
\newcommand{\eeq}{\end{equation}}
\newcommand{\dis}{\displaystyle}

\newtheorem{lemma}{Lemma}%[section]
\newtheorem{thm}[lemma]{Theorem}
\newtheorem{prop}[lemma]{Proposition}

\theoremstyle{definition}
\newtheorem{defi}[lemma]{Definition}
\newtheorem{rmk}[lemma]{Remark}

\newcommand{\cred}{\textcolor{red}}
\newcommand{\cblue}{\textcolor{blue}}

\begin{document}

\title[Traveling waves for discontinuous reaction-diffusion-advection equations]{Traveling waves for highly degenerate and singular reaction-diffusion-advection equations with discontinuous coefficients}

\author[U. Guarnotta]{Umberto Guarnotta}
\address[U. Guarnotta]{Dipartimento di Ingegneria Industriale e Scienze Matematiche, Università Politecnica delle Marche, Via Brecce Bianche 12,
60131 Ancona, Italy}
\email{u.guarnotta@univpm.it}

\author[C. Marcelli]{Cristina Marcelli}
\address[C. Marcelli]{Dipartimento di Ingegneria Industriale e Scienze Matematiche, Università Politecnica delle Marche, Via Brecce Bianche 12,
60131 Ancona, Italy}
\email{c.marcelli@staff.univpm.it}

\begin{abstract}
Sufficient conditions for either existence or non-existence of traveling wave solutions for a general quasi-linear reaction-diffusion-convection equation, possibly highly degenerate or singular, with discontinuous coefficients are furnished. Under an additional hypothesis on the convection term, the set of admissible wave speeds is characterized in terms of the minimum wave speed, which is estimated through a double-sided bound.
\end{abstract}

\maketitle

{
\let\thefootnote\relax
\footnote{{\bf{MSC 2020}}: 35C07; 35K57; 35K59; 35K65.}
\footnote{{\bf{Keywords}}: Traveling waves; Degenerate equations; Singular equations; Quasi-linear parabolic equations; Discontinuous coefficients; Singular terms; Gamma-convergence.}
\footnote{\Letter \quad Corresponding author: Umberto Guarnotta (u.guarnotta@univpm.it).}
}
\setcounter{footnote}{0}

\section{Introduction}

\bigskip
As it is well-known, front-type solutions for reaction equations are a research area covered by many papers. Quite a wide literature is available on the existence and the properties of {\it traveling wave solutions} (for short t.w.s.'s) for many types of reaction-diffusion equations: monostable or bistable, degenerate or doubly degenerate, with advection or accumulation effects, or with changing-sign diffusivities, governed by the $p$-Laplacian or other type of differential operators (see, e.g. the recent papers \cite{AV1,AV2,BCM,CMP,CMS,GS2,Ga}). 

Such a great variety of the dynamics investigated concerns equations with continuous coefficients, and consequently regular t.w.s.'s. Nevertheless, there are some models involving discontinuous coefficients (see, e.g., \cite{SZW} for a model with a discontinuous diffusivity, and \cite{ARV, DHT, FN, Te, Va} for models with a reaction term depending on the Heaviside function).
This motivates the interest in studying t.w.s.'s for reaction-diffusion equations with discontinuous coefficients.

In \cite{DJKZ} (see also \cite{DZ}) the authors considered the monostable equation
\[u_t+f(u)u_x=(d(u)|u_x|^{p-2}u_x)_x + h(u)\]
with \(p>1\), where the reaction term \(h\) is continuous, but the diffusivity \(d\) is lower-semicontinuous and the convective term \(f\) is piecewise continuous. Both \(d\) and \(f\) may have jump-type discontinuities. The authors proved the existence of infinitely many t.w.s.'s having speed \(c\) in a halfline \([c^*,+\infty)\) and provided an estimate for the unknown threshold value \(c^*\).

According to our knowledge, no papers concern the study of t.w.s.'s for general equations involving discontinuous reaction terms \(h\) too, and this is the main goal of the present paper.

Indeed, 
we consider a very general monostable reaction-diffusion-convection equation involving also a nonlinear accumulation term \(g\), that is,
\begin{equation}
\label{pdeprob}
f(u)u_x + g(u)u_t = (d(u)|u_x|^{p-2}u_x)_x + h(u), \quad (x,t)\in\R^2,
\end{equation}
with $p>1$, where all the coefficients \(f,g,d,h\) are just piecewise continuous functions, having at most a finite number of discontinuities.

More precisely,
let $\hat{C}[0,1]$ [resp., $\hat{C}(0,1)$] be the set of the piecewise continuous functions $\mu:[0,1]\to\R$ [resp., $\mu:(0,1)\to \R$] having a finite set $\D(\mu):=\{\gamma_1,\ldots,\gamma_n\}\subset (0,1)$ of discontinuity points. In what follows, we do not consider removable discontinuities, that is, points $\xi_0\in(0,1)$ such that $\mu(\xi)\to l\in\R$ as $\xi\to\xi_0$ but $l\neq\mu(\xi_0)$, since our solutions will be strictly monotone and the value of a coefficient at a single point does not influence the dynamics. In other words, we will identify  functions which differ, at most, in a finite number of points.

Throughout the paper, we assume that
\begin{enumerate}[label={$({\rm H}_\arabic*)$},ref={${\rm H}_\arabic*$}]
\item\label{ip:hatC} $f,g,h\in\hat{C}[0,1]\cap L^\infty(0,1)$, $d\in \hat C(0,1)$
\item\label{ip:d}  
 \(\dis\inf_{\xi\in[a,b]}d(\xi)>0\) and   \(\dis\sup_{\xi\in[a,b]}d(\xi)<+\infty\) whenever \(0<a<b<1\).
\item\label{ip:acca} 
 $h(0)=h(1)=0$ and 
 \(\dis\inf_{\xi\in[a,b]}h(\xi)>0\)  whenever \(0<a<b<1\).
\item\label{ip:kappa} $\kappa:=d^{\frac{1}{p-1}}h\in L^1(0,1)$.
\end{enumerate}
Notice that \(d\) is allowed to be unbounded near 0 and 1, or to vanish at 0 and 1. When \(d(0^+)=0\) and/or \(d(1^-)=0\), the equation  \eqref{pdeprob} is said {\it degenerate} or {\it doubly degenerate}. Analogously, it is said {\it singular} or {\it doubly singular} if $d(u)\to+\infty$ when $u\to0^+$ and/or $u\to1^-$.

%Concerning the reaction term $h:[0,1]\to\R$, we will suppose that it is of {\it monostable} type, that is $h>0$ in $(0,1)$ and $h(0)=h(1)=0$

When searching  for  t.w.s.'s, that is, solutions $u(x,t):=v(x-ct)$, for some one-variable function \(v\) (the wave profile) and some $c\in\R$  (the wave speed), 
 the wave profile $v=v(z)$ solves the ordinary differential equation
\begin{equation}
\label{odeprob0}
(d(v)|v'|^{p-2}v')'+(cg(v)-f(v))v'+h(v)=0, \quad v\in[0,1].
\end{equation}
More specifically, we are interested in the heteroclinic solutions to \eqref{odeprob0} joining the equilibria $v=1$ and $v=0$, i.e., the solutions to
\begin{equation}
\label{odeprob}
\left\{
\begin{alignedat}{1}
&(d(v)|v'|^{p-2}v')'+(cg(v)-f(v))v'+h(v)=0, \quad v\in[0,1], \\
&v(-\infty)=1,\quad v(+\infty)=0.
\end{alignedat}
\right.
\end{equation}
Due to the presence of possible discontinuities, we cannot expect the solutions of \eqref{odeprob} to be regular everywhere, but we will search for solutions continuously differentiable at each point where \(0<v(z)<1\), with the exception, at most, of the discontinuity points of the diffusivity \(d\), satisfying equation \eqref{odeprob} at the points where all the coefficients are continuous (see Definition \ref{defsol}, Proposition \ref{prop:nec}, and Remark \ref{rem:differ} for a detailed discussion about this matter).

In order to state our main result, we introduce the following notations, by putting
% \beq \kappa(\xi):=d(\xi)^{\frac{1}{p-1}}h(\xi) \quad \mbox{for all} \;\;\xi\in(0,1) \label{def:kappa} \eeq
% and
\begin{equation}
\label{notation}
F_0:=\sup_{\xi\in(0,1)}\fint_0^\xi f(\tau)\dtau, \quad G_0:=\inf_{\xi\in(0,1)}\fint_0^\xi g(\tau)\dtau, \quad  K_0:=\sup_{\xi\in(0,1)}\fint_0^\xi \frac{\kappa(\tau)}{\tau^{\frac{1}{p-1}}}\dtau,
\eeq
\beq
\ell_p:=\liminf_{\xi\to0^+}\frac{\kappa(\xi)^{p-1}}{\xi}\ , \quad L_p:=\limsup_{\xi\to0^+}\frac{\kappa(\xi)^{p-1}}{\xi}.
\label{notation2}
\end{equation}
We will consider the set
\beq \label{def:Gamma-c}
\C:=\{c\in\R: \, \mbox{there exists a t.w.s. of speed } \, c\}.
\eeq

The following theorem is an immediate consequence of Theorems \ref{t:mainthm1} and \ref{t:mainthm2} in Section 5 (see also Remark \ref{finalrmk}).

\begin{thm}
\label{mainthm}
Suppose that $G_0>0$. Then, if 
$\ell_p=+\infty$, no t.w.s. exists. Instead, if \(L_p<+\infty\), then $\C=[c^*,+\infty)$ for some $c^*\in\R$ fulfilling
\begin{equation}
\label{threshold}
\frac{f(0)}{g(0)}+\frac{p'}{g(0)}(p-1)^{\frac{1}{p}}\ell_p^{\frac{1}{p}} \leq c^* \leq \frac{F_0}{G_0}+\frac{p'}{G_0}(p-1)^{\frac{1}{p}}K_0^{\frac{1}{p'}}.
\end{equation}
\end{thm}
Incidentally, we notice that \eqref{threshold} is well defined, since $\displaystyle{g(0)=\lim_{\xi\to0^+}\fint_0^\xi g(\tau)\dtau\geq G_0>0}$. Actually (see Theorems \ref{t:mainthm1} and \ref{t:mainthm2}, as well as Remark \ref{finalrmk}), the assumption $G_0>0$ can be relaxed by supposing $g(0)>0$ and
\begin{equation}
\label{implement1}
\int_0^\xi g(\tau)\dtau\geq0 \quad \mbox{for all} \;\; \xi\in[0,1],
\end{equation}
provided that \eqref{threshold} is replaced by the more general conditions
\begin{equation*}
%\label{implement2}
\begin{aligned}
\inf_{\xi\in(0,1)} \fint_0^\xi (c^*g(\tau)-f(\tau)) \dtau &\leq p'(p-1)^{\frac{1}{p}}\left(\sup_{\xi\in(0,1)} \fint_0^\xi \frac{\kappa(\tau)}{\tau^{\frac{1}{p-1}}}\dtau\right)^{\frac{1}{p'}}, \\
c^* g(0)-f(0) &\geq p'(p-1)^{\frac{1}{p}}\ell_p^{\frac{1}{p}}.
\end{aligned}
\end{equation*}
Finally, if one removes also condition \eqref{implement1}, then it is possible to show that \(\mathcal C\) admits minimum \(c^*\) but in this case we cannot state  \(\mathcal C\) is a half-line (see Theorem \ref{t:mainthm2}).

%\cred{Letturatura.}

%This paper was motivated by \cite{CMP} (see also \cite{M}) and \cite{DJKZ}: the former treats a semi-linear reaction-diffusion-convection equation having continuous coefficients and a general convection term, while the latter deals with \eqref{pdeprob} when $g\equiv 1$ and only $f,d$ may be discontinuous. 

Our results generalize the ones in \cite{DJKZ} in various directions:
\begin{itemize}
\item the generality of the reaction-diffusion equation, which involves also an accumulation term \(g\), possibily discontinuous and not necessarily positive;
\item we allow all the terms to be discontinuous (in particular the reaction term \(h\)), without restrictions on the type of discontinuities;
\item the estimate \eqref{threshold} for the threshold value \(c^*\) is finer that the one in \cite[Theorem 2.1]{DJKZ}, according to Remark \ref{finalrmk};
\item the assumptions required in the existence and non-existence results are more general.
\end{itemize}

\noindent
More in detail, as for the last point, we underline that by requiring assumption \eqref{ip:kappa}, we allow not only the diffusion \(d\), but also the product \(\kappa\) to be unbounded near the equilibrium 1. According to our knowledge, this situation is new even in the classical case.

\medskip
In the present paper we do not investigate the classification of the t.w.s.'s, that is, if they attains the equilibria at finite times (with a non-zero slope), distinguishing between {\it classical} and {\it sharp} ones (see, e.g., \cite{CMP}), since possible discontinuities in the interior of the interval \([0,1]\) do not interfer with the behavior of the t.w.s. at the equilibria.
So, the classification of the t.w.s.'s in the case of discontinuous coefficients does not differ from that of the continuous ones.  Instead, the case when  \(\kappa\), or simply \(d\), is unbounded near 1 is completely new (also for the classical case) and might deserve to be investigated.

% It is worth noticing that \cite{M} partially addresses the first issue, but not the second one. \cblue{Anyway, we highlight that, unlike \cite{CMP}, the present paper does not deal with the classification of t.w.s.'s to \eqref{pdeprob}, which is a topic of independent interest and can be analyzed hopefully arguing as in \cite[Section 5]{CMP}, since each coefficient of \eqref{pdeprob} is continuous near the equilibrium states $0$ and $1$.}

\bigskip

We proceed in this way. After giving, in Section 2, the definition of solution to \eqref{odeprob}, coherently with \cite{CMP} and \cite{DJKZ}, we analyze its main properties and we show how \eqref{odeprob} reduces to a first-order singular boundary value problem (BVP) in Section 3. En passant, we prove a local a priori estimate from below, which is crucial to handle the singularity. Then, in Section 4, we regularize the singular BVP by introducing the concept of $\eps$-regularization, reminiscent of the classical linear interpolation methods. Then, we use the inf-stability of the $\Gamma$-convergence, discussed in the same section, to ensure that the existence conditions in \cite{M}, involving the quantities appearing in \eqref{notation}, are preserved through $\eps$-regularization at the limit $\eps\to0^+$. Finally, in Section 5, we give sufficient conditions ensuring either existence or non-existence of t.w.s.'s; then, assuming also \eqref{implement1}, we present a structure theorem for $\C$.

\section{Solutions and their properties}

\bigskip

This section is devoted to giving the definition of solution to problem \eqref{odeprob} and studying the main properties of the solutions.

First of all, let us introduce the following notation. Set

\begin{equation*}
\Theta:= \D(d)\cup\D(f)\cup\D(g)\cup\D(h) \ \ \text{and} \ \ \Theta^*:=\Theta\cup\{0,1\}. %\label{def:Theta}
\end{equation*}
Notice that if \(v:\R\to [0,1]\) is continuous then both \(v^{-1}(\Theta)\) and \(v^{-1}(\Theta^*)\) are closed, since \(\Theta\) and \(\Theta^*\) are closed.

\begin{defi}\label{defsol}
A continuous function \(v:\R\to [0,1]\) is  a solution of  \eqref{odeprob} if \( v\in C^1(\R\setminus v^{-1}(\Theta^*))\), satisfies the boundary conditions at \(\pm \infty\), and there exists a continuous function \(\Phi_v:\R\to \R\) such that:
\begin{enumerate}[label={{\rm (\roman*)}}]
\item\label{defa}  \(\Phi_v(z)=d(v(z))|v'(z)|^{p-2}v'(z)\) in each open interval where \(v\) is $C^1$;
\item\label{defnew} \(\dis\lim_{z\to \pm \infty} \Phi_v(z)=0\);
\item\label{defb} \(\Phi_v(z)=0\) whenever \(v(z)=0\) or \(v(z)=1\);
\item \label{defc} the following integral formulation of \eqref{odeprob} holds for any \(z_1,z_2\in \R\):
\beq 
\Phi_v(z_2)-\Phi_v(z_1)+\int_{v(z_1)}^{v(z_2)} (cg(\xi)-f(\xi)) \dxi + \int_{z_1}^{z_2} h(v(z)) \dz=0.\label{eq:integr}
\eeq
%where \(F(\xi):=\dis\int_0^\xi f(\tau) \dtau\) and %\(G(\xi):=\dis\int_0^\xi g(\tau) \dtau\).

\end{enumerate}

\end{defi}

\medskip
\noindent
The following result summarizes the main properties of the solutions to problem \eqref{odeprob}.

\medskip

\begin{prop} \label{prop:nec}
Let \(v\) be a solution of problem \eqref{odeprob}. Then, put 

\beq
I_v:=\{z\in \R : \ 0<v(z) <1\}, \label{def:Iv}
\eeq 
we have that \(I_v\) is an open (bounded or unbounded) interval and \(v\) is strictly decreasing in \(I_v\), with \(v'(z)<0\) for every \(z\in I_v\setminus v^{-1}(\Theta)\).
%Moreover,  for every \(z\in  v^{-1}(\Theta)\) the one-sided derivatives \(v'(z^-)\) and \(v'(z^+)\) exist and are finite.

Furthermore, $\Phi_v(z)<0$ for every $z\in I_v$, \(\Phi_v\) is $C^1$ in \(I_v\setminus v^{-1}(\Theta)\), and satisfies \beq
\Phi_v'(z)+(cg(v(z))-f(v(z)))v'(z)+h(v(z))=0 \quad \text{for every } z\in I_v\setminus v^{-1}(\Theta). \label{eq:differ}
\eeq

%Finally, if \(d\) is bounded then and %\(\dis\lim_{z\to \pm \infty} \Phi_v(z)=0\).

\end{prop}

\medskip
\begin{proof}
We split the proof into various steps.

\medskip

\smallskip
\noindent

\medskip
{\bf Claim 1.} The solution \(v\) cannot be constant in any interval \((a,b)\subseteq I_v\).

\smallskip
Indeed, if \(v\) is constant in an interval \((a,b)\subseteq I_v\), then by the definition of solution we have \(\Phi_v(z)=d(v(z))|v'(z)|^{p-2}v'(z)=0\) for each \(z\in (a,b)\), and by \eqref{eq:integr} we derive \(h(v(z))=0\) for every \(z\in (a,b)\), contradicting the positivity of the function \(h\) in \((0,1)\).

\medskip
{\bf Claim 2.} If \(v(z_1)=v(z_2)\) for some \(z_1,z_2 \in I_v\), \(z_1<z_2\), then \(v(z)\ge v(z_1)\) for every \(z\in (z_1,z_2)\cap I_v\).

\smallskip
Let \(z_1,z_2\in I_v\) be such that \(v(z_1)=v(z_2)\). Assume by contradiction that \(v(z^*)<v(z_1)\) for some \(z^*\in (z_1,z_2)\cap I_v\). Since \(\Theta\) is a finite set, there exists a value \(\bar\xi\in (v(z^*),v(z_1))\setminus \Theta\). Put 
\[ \zeta_1:= \sup\{z>z_1 \ : v(\zeta)\ge\bar\xi \ \text{ for every } \zeta\in (z_1,z)\}\]
\[\zeta_2:= \inf\{z<z_2 \ : v(\zeta)\ge \bar\xi \ \text{ for every } \zeta\in (z,z_2)\}.\]
Of course, \(v(\zeta_1)=v(\zeta_2)=\bar\xi\) and  \(\zeta_1<\zeta_2\). Moreover,  since \(\bar\xi \not \in \Theta\), \(v\) is $C^1$ in a neighborhood of \(\zeta_1, \zeta_2\), with \(v'(\zeta_1)\le 0\) and \(v'(\zeta_2)\ge 0\), by the definition of \(\zeta_1,\zeta_2\).
On the other hand, by \eqref{eq:integr}, we have
\begin{align*}
d(\bar\xi) (|v'(\zeta_2)|^{p-2}v'(\zeta_2)-|v'(\zeta_1)|^{p-2}v'(\zeta_1))&= d(v(\zeta_2))|v'(\zeta_2)|^{p-2}v'(\zeta_2) - d(v(\zeta_1))|v'(\zeta_1)|^{p-2}v'(\zeta_1) \\
&= \Phi_v(\zeta_2)-\Phi_v(\zeta_1)=-\int_{\zeta_1}^{\zeta_2} h(v(z)) \dz <0
\end{align*}
implying \(v'(\zeta_2)<v'(\zeta_1)\), a contradiction.

\medskip

%{\bf Claim 3.} If \(v(z_1)=v(z_2)\) for some \(z_1,z_2 \in I_v\) with \(z_1<z_2\), then \(v(z)> v(z_1)\) for every \(z\in (z_1,z_2)\).

%\smallskip
%Assume \(v(z^*)=v(z_1)\) for some \(z^*\in (z_1,z_2)\).By claims 1 and 2 we deduce the existence of \(\zeta_1 \in (z_1,z^*)\) and \(\zeta_2\in (z^*,z_2)\) such that \(v(\zeta_1), v(\zeta_2)> v(z^*)\). Since \(\Theta\) is finite, we can choose a value \(\xi^* \in (0,1)\setminus \Theta\), such that \(v(z^*)<\xi^*<\min\{v(\zeta_1), v(\zeta_2)\}\). So, there exist \(\zeta_1^*\in (\zeta_1,z^*)\) and \(\zeta_2^*\in (z^*,\zeta_2)\) such that \(v(\zeta_1^*)=v(\zeta_2^*)=\xi^*> v(z^*)\), in contradiction with Claim 2.

\medskip
{\bf Claim 3.} The solution \(v\) is strictly decreasing in \(I_v\) and \(I_v\) is an open interval.
\smallskip

Assume, by contradiction, that \(v(z_1)< v(z_2)\) for some 
\(z_1,z_2\in I_v\), with \(z_1<z_2\). Then, since \(v(-\infty)=1\), we infer that there exists a point \(\zeta_1 <z_1\) such that \(v(\zeta_1)=v(z_2)\), in contrast to Claim 2. So, \(v\) is decreasing in \(I_v\). Moreover, by Claim 1 we deduce that \(v\) is strictly decreasing in \(I_v\).
Then, \(I_v\) consists of a unique (open) interval.

\medskip
{\bf Claim 4.} \(\Phi_v\) is \(C^1\) in \(I_v\setminus v^{-1}(\Theta)\) and equation \eqref{eq:differ}
 holds true.
\smallskip

Fix \(z\in I_v\setminus v^{-1}(\Theta)\) and let \(\delta>0\) be such that \(z+\varepsilon\in I_v\setminus v^{-1}(\Theta) \) if \(|\varepsilon|<\delta\) (this is possible since \(I_v\setminus v^{-1}(\Theta)\) is an open set).
So, consider equation \(\eqref{eq:integr}\) with \(z_1=z\) and \(z_2=z+\varepsilon\). Dividing by \(\varepsilon\ne 0\) and passing to the limit as \(\varepsilon\to 0\) we get that \(\Phi_v\) is differentiable at \(z\) and
\[ \Phi_v'(z) + [c(g(v(z))- f(v(z))]v'(z) + h(v(z))=0.\]
Hence, \(\Phi_v\) is \(C^1\) in \(I_v\setminus v^{-1}(\Theta)\).

\medskip
{\bf Claim 5.} For every \(z\in I_v \setminus v^{-1}(\Theta)\) we have \(v'(z)<0\) and \(\Phi_v(z)<0\).
\smallskip

Indeed, assume by contradiction \(v'(z_0)=0\) for some \(z_0\in I_v\setminus v^{-1}(\Theta)\). Then, since $v\in C^1$ in a neighborhood of $z_0$, by equation 
\eqref{eq:differ} we get \(\Phi_v'(z_0)=(d(v(z_0))|v'(z_0)|^{p-2}v'(z_0))'=-h(v(z_0))<0\), with \(\Phi_v(z_0)=0\), so we have \(\Phi_v(z)=d(v(z))|v'(z)|^{p-2}v'(z)>0\) in a left neighborhood of \(z_0\), in contradiction to Claim 3. So, \(v'(z)<0\) and by Definition \ref{defsol}\ref{defa} we also have \(\Phi_v(z)<0\).

% \medskip
% {\bf Claim 6.} for every \(z\in v^{-1}(\Theta)\) there exist, finite, the one-sided derivatives \(v'(z^-), v'(z^+)\).
% \smallskip

% Fix \(z_0\in v^{-1}(\Theta)\). Since \(v\) is strictly decreasing in \(I_v\) and \(\Theta\) is a finite set, we have that \(v\) is differentiable in a punctured neighborhood of \(z_0\). Then
% \[ \lim_{z\to z_0^\pm } v'(z) = -\lim_{z\to z_0^\pm } \left|\frac{\Phi_v(z)}{d(v(z))}\right|^{\frac{1}{p-1}} = -\left|\frac{\Phi_v(z_0)}{d(v(z_0)^\mp)}\right|^{\frac{1}{p-1}} \in \R.\]

\medskip
{\bf Claim 6.} $\Phi_v(z)<0$ for every \(z\in I_v\).

\smallskip
 Since \(\Theta\) is finite, by Claim 3 also $v^{-1}(\Theta)$ is finite. So, Claim 5 ensures \(\Phi_v(z)\le 0\) for every \(z\in \R\). Assume by contradiction \(\Phi_v(z_0)=0\) for some \(z_0\in v^{-1}(\Theta)\). Since \(v\) is strictly decreasing in \(I_v\) and \(\Theta\) is a finite set, we have that \(v\) is $C^1$ in a punctured neighborhood of \(z_0\). Hence,
\[ 0=\Phi_v(z_0)=\lim_{z\to z_0} \Phi_v(z)=-\lim_{z\to z_0^\pm} d(v(z))|v'(z)|^{p-1}. \]
Therefore, since \(\displaystyle{\liminf_{\xi\to v(z_0)}d(\xi)>0}\), we deduce \(v'(z_0^+)=v'(z_0^-)=0\), implying that \(v\) is differentiable at \(z_0\) with \(v'(z_0)=0\). Hence, chosen a number \(\delta>0\) such that \([z_0-\delta, z_0+\delta]\subseteq (0,1)\) and put \(\mu:=\inf_{[z_0-\delta,z_0+\delta]} h(z)>0\), by \eqref{eq:integr} we have
\begin{align*}
\limsup_{\varepsilon \to 0^+} \frac{-\Phi_v(z_0-\varepsilon)}{\varepsilon} &=\limsup_{\varepsilon \to 0^+} \frac{1}{\varepsilon} \left(\Phi_v(z_0)- \Phi_v(z_0-\varepsilon) \right)  \\ 
&= \limsup_{\varepsilon \to 0^+}\frac{1}{\varepsilon} \left( \int_{v(z_0-\varepsilon)}^{v(z_0)} (f(\xi)-c g(\xi))\dxi  -\int_{z_0-\varepsilon}^{z_0} h(v(\zeta)) \dzeta \right) \\ &=\limsup_{\varepsilon \to 0^+}\left(\frac{v(z_0)-v(z_0-\varepsilon)}{\varepsilon}  \fint_{v(z_0-\varepsilon)}^{v(z_0)} (f(\xi)-c g(\xi)) \dxi  -\fint_{z_0-\varepsilon}^{z_0} h(v(\zeta)) \dzeta \right)\\ &\le -\mu<0.
\end{align*}

\noindent
Therefore, \(\Phi(z)>0\) in a left neighborhood of \(z_0\), a contradiction. 

\end{proof}

\begin{rmk} Taking account of the properties proved in the previous proposition, it is  easy to check that  Definition  \ref{defsol} reduces to the classical one when \(\Theta=\emptyset\). Indeed, in this case \(\Theta^*=\{0,1\}\) and then \(v\in C^1(I_v)\). So \(\Phi_v(z) = d(v(z))|v'(z)|^{p-2} v'(z)\) in \(I_v\), with \(\Phi_v\in C^1(I_v)\) and
\[ (d(v(z))|v'(z)|^{p-2} v'(z))'+ (cg(v(z))-f(v(z)))v'(z) + h(v(z)) =0 \]
for every \(z\in I_v\). Finally,
\[ \lim_{z\to \pm \infty} d(v(z))|v'(z)|^{p-2} v'(z)=0. \]
\end{rmk}

\begin{rmk}\label{rem:differ}
 In light of Proposition \ref{prop:nec}, we see that the set \(v^{-1}(\Theta)\) is finite, so there are, at most, a finite number of points where \(v\) is not differentiable. %, with \(v'(z^+), v'(z^-) \le 0\) (cfr. Claim 6 of Proposition \ref{prop:nec}).
However, note that actually \(v\) is $C^1$ in \(I_v\setminus \D(d)\). Indeed, if \(z_0\in I_v\setminus \D(d)\), since \(v\) is $C^1$ in a punctured neighborhood of \(z_0\), by Definition \ref{defsol}\ref{defa} there exists, finite, the limit
\[ \dis  \lim_{z\to z_0} v'(z) = -\lim_{z\to z_0} \left|\frac{\Phi_v(z)}{d(v(z))}\right|^{\frac{1}{p-1}}=-\left|\frac{\Phi_v(z_0)}{d(v(z_0))}\right|^{\frac{1}{p-1}}<0.\]
%Moreover, since  $v'(z)<0$ in $I_v\setminus %v^{-1}(\D(d))$, then by Definition %\ref{defsol}\ref{defa} we have
%\begin{equation}
%\label{Phi}
%\Phi_v(z)=-d(v(z))|v'(z)|^{p-1}.
%\end{equation}

Finally, since \(I_v\) is an interval, we find that 
 if \(v(z_*)=1\) for some \(z_*\in \R\), then \(v(z)=1\) for all \(z\le z_*\); similarly,
 if \(v(z^*)=0\) for some \(z^*\in \R\), then \(v(z)=0\) for all \(z\ge z^*\).

Therefore, put 
\beq \Gamma_v := v^{-1}(\D(d)) \ , \quad  \Gamma_v^\circ = v^{-1}(0)\cup v^{-1}(1), \quad   \Lambda_v:= \Gamma_v\cup \partial\Gamma_v^\circ,
\label{def:Gamma} \eeq
 $\Lambda_v$ is a finite set
and \(v\in C^1(\R\setminus \Lambda_v)\). 
%\label{rem:diff}

\noindent
We underline that if \(v\) attains one or both the values 0 and 1 at finite \(z\), that is if \(\Gamma_v^\circ\ne \emptyset\), then in general \(v\) is not differentiable at such points and the t.w.s. is said to be {\em sharp}. This phenomenon does not depend on the presence of discontinuities of the coefficients of the equation, but it occurs when \(d\) vanishes at 0 or at 1 (degenerate or doubly degenerate equations). As we clarified in the Introduction, we do not investigate this topic in the present paper.
\end{rmk} 

\begin{rmk} \label{rem:AC}
If \(v\) is a solution of problem \eqref{odeprob}, due to the monotonicity proved in Proposition \ref{prop:nec}, we have that \(v'\in L^1_{\text{loc}}(\R)\) so, since \(v\in C^0(\R)\cap C^1(\R\setminus \Lambda_v)\), we deduce \(v\in W^{1,1}_{\text{loc}}(\R)\). 
\newline
Similarly, equation \eqref{eq:differ} 
implies that \(\Phi_v'\in L^1_{\text{loc}}(\R)\) and \(\Phi_v\) is \(C^1\) in \(I_v\setminus v^{-1}(\Theta)\). Hence, also \(\Phi_v\in W^{1,1}_{\text{loc}}(\R)\).
\end{rmk}

\begin{prop}
    If \(d\) is bounded in \((0,1)\), then condition \ref{defnew} in Definition \ref{defsol} is a consequence of the other conditions \ref{defa}, \ref{defb} and \ref{defc}.
\end{prop}

\begin{proof}
First, notice that condition \ref{defnew} has not been used in the proof of Proposition \ref{prop:nec}. 

Put \(a_v:=\inf I_v\) and \(b_v:=\sup I_v\) (see \eqref{def:Iv}). If either \(a_v\) or \(b_v\) is finite, condition \ref{defnew} is a consequence of \ref{defb} and the monotonicity of \(v\). Instead, if \(b_v=+\infty\) then by \eqref{eq:integr} we have, for any \(z_1, z\) with \(z_1<z\),
\[ \Phi_v(z)- \Phi_v(z_1)= 
\int_{v(z)}^{v(z_1)} (cg(\xi)-f(\xi)) \dxi
-\int_{z_1}^z h(v(\zeta)) \dzeta. \]
Therefore, there exists (finite or not) the limit
\[ L:= \lim_{z\to +\infty} \Phi_v(z)= \Phi_v(z_1)+\int_0^{v(z_1)} (cg(\xi)-f(\xi))\dxi-\int_{z_1}^{+\infty} h(v(\zeta)) \dzeta. \]
Since \(\Phi_v(z)<0\) for every \(z\in I_v\), we infer \(L\le 0\). If \(L< 0\) then, put  
\(M:=\dis \sup_{\xi \in (0,1)} d(\xi)\),
 we infer
\[ \limsup_{z\to +\infty} v'(z)= -\liminf_{z\to +\infty} \left|\frac{\Phi_v(z)}{d(v(z))}\right|^{\frac{1}{p-1}}\leq -\left(\frac{|L|}{M}\right)^{\frac{1}{p-1}}<0,\]
in contrast to the boundedness of \(v\). Thus, \(L=0\). The same argument works when \(a_v=-\infty\).
\end{proof}

\section{Reduction to a first-order problem}

As usual in the study of t.w.s.'s, 
problem \eqref{odeprob} can be reduced to a boundary value problem for a singular first-order ODE, as the next result shows.

% From now on we put 
% \beq \kappa(\xi):= h(\xi)d(\xi)^{\frac{1}{p-1}}.
% \label{def:kappa}
% \eeq

\begin{prop}\label{prop:equiv}
Let the function \(\kappa\) be defined in \eqref{ip:kappa}. Then problem \eqref{odeprob} is equivalent to
\begin{equation}
\label{odeprob2}
\left\{
\begin{alignedat}{2}
\dot y &= cg(\xi)-f(\xi)-\frac{\kappa(\xi)}{y^{\frac{1}{p-1}}} \quad &&\mbox{in} \;\; (0,1)\setminus \Theta, \\
y&>0 \quad &&\mbox{in} \;\; (0,1), \\
y(0^+)&=y(1^-)=0,
\end{alignedat}
\right.
\end{equation}
in the following sense:

\begin{itemize}
\item if $v\in C(\R)\cap C^1(\R\setminus\Lambda_v)$ (see \eqref{def:Gamma}) is a solution to \eqref{odeprob}, then $y(\xi):=-\Phi_v(v^{-1}(\xi))$ is in $C(0,1)\cap C^1((0,1)\setminus\Theta)$ and solves \eqref{odeprob2};

\item if $y\in C(0,1)\cap C^1((0,1)\setminus\Theta)$ is a solution to \eqref{odeprob2} then the function $w:(0,1)\to\R$ defined as
$$ w(\xi):=-\int_{\frac{1}{2}}^\xi \left(\frac{d(\tau)}{y(\tau)}\right)^{\frac{1}{p-1}} \dtau \quad \mbox{for all} \;\; \xi\in(0,1) $$
is invertible in a suitable interval $(a,b)$ (possibly unbounded), and its inverse $w^{-1}:(a,b)\to(0,1)$ can be extended to a solution of \eqref{odeprob}.
\end{itemize}
\end{prop}

\begin{proof}
%We set $H(\xi):=cg(\xi)-f(\xi)$ and
%$K(\xi):=D(\xi)h(\xi)$ for all $\xi\in(0,1)$. 
Pick any $v\in C(\R)\cap C^1(\R\setminus\Lambda_v)$ solution to \eqref{odeprob}. Then, for every $\xi\in(0,1)\setminus\Theta$, one has $v^{-1}(\xi)\in I_v\setminus v^{-1}(\Theta)$. Hence, \eqref{eq:differ} with $z=v^{-1}(\xi)$  ensures
$$ \Phi_v'(v^{-1}(\xi))+(cg(\xi)-f(\xi))v'(v^{-1}(\xi))+h(\xi)=0 \quad \mbox{for all}\;\; \xi\in(0,1)\setminus\Theta. $$
Dividing by $-v'(v^{-1}(\xi))$, which is positive, entails
\beq \label{eq:doty} -\frac{\Phi_v'(v^{-1}(\xi))}{v'(v^{-1}(\xi))}-(cg(\xi)-f(\xi))-\frac{h(\xi)}{v'(v^{-1}(\xi))}=0 \quad \mbox{for all}\;\; \xi\in(0,1)\setminus\Theta. \eeq
Setting $y(\xi):=-\Phi_v(v^{-1}(\xi))$ for all $\xi\in(0,1) $, it turns out that $y\in C(0,1)\cap C^1((0,1)\setminus\Theta)$ with $\dot y(\xi)=-\frac{\Phi_v'(v^{-1}(\xi))}{v'(v^{-1}(\xi))}$. Therefore, by \eqref{eq:doty} we get
$$ \dot y(\xi)-(cg(\xi)-f(\xi))+\frac{\kappa(\xi)}{y(\xi)^{\frac{1}{p-1}}}=0 \quad \mbox{for all}\;\; \xi\in(0,1)\setminus\Theta. $$
Moreover, since $\Phi_v(z)<0$ for all $z\in I_v$ (see Proposition \ref{prop:nec}), then $y(\xi)>0$ for all $\xi\in(0,1)$. Finally, by Definition \ref{defsol}\ref{defnew}--\ref{defb}, we have $\Phi_v(v^{-1}(\xi))\to 0$ as $\xi\to 0^+$ (resp., $\xi\to 1^-$), so $y(0^+)=0$ (resp., $y(1^-)=0$).

\medskip
Now let $y\in C(0,1)\cap C^1((0,1)\setminus\Theta)$ be a solution to \eqref{odeprob2}. Since $y(\xi)>0$ for all $\xi\in(0,1)$ and $d^{\frac{1}{p-1}}=\frac{\kappa}{h}\in L^1(\frac{1}{2},\xi)$ for all $\xi\in(0,1)$, due to \eqref{ip:acca} and \eqref{ip:kappa}, then $w$ is well-defined and strictly decreasing in $(0,1)$. Moreover, $w$ is differentiable at every \(\xi\in (0,1)\setminus \D(d)\), with $w'(\xi)=-\left(\frac{d(\xi)}{y(\xi)}\right)^{\frac{1}{p-1}}<0$. 

Set $a:=w(1^-)$ and $b:=w(0^+)$ (which may be one or both infinite), and consider 
\beq \label{eq:inverse} v(z):=w^{-1}(z)\ , \quad \Phi_v(z):= -y(v(z))\ , \quad \text{for all }  z\in(a,b).\eeq 
If either \(a\) or \(b\) is finite, we extend the functions \(v, \Phi_v\) as constants, i.e. we put
\beq v(z)\equiv 1 \ \text{ in } (-\infty,a]; \quad v(z)\equiv 0\ \text{ in } [b,+\infty); \quad \Phi_v(z)\equiv 0 \ \text{ in } \R\setminus (a,b).  \label{eq:exten}\eeq

Of course, $v$ and \(\Phi_v\) are continuous in $\R$ and \(v\) is strictly decreasing in \(I_v:=(a,b)\). Moreover, $v$ is 
differentiable at any $z\in I_v\setminus v^{-1}(\D(d))$, with $v'(z)=-\left(\frac{y(v(z))}{d(v(z))}\right)^{\frac{1}{p-1}}<0$. Hence, we have 
\[\Phi_v(z)= - y(v(z))= -d(v(z))|v'(z)|^{p-1}<0 \quad \text{for every } z \in I_v\setminus v^{-1}(\D(d)).
\]
On the other hand, if \(z\in v^{-1}(\D(d))\), then \(v\) is not \(C^1\) in a neighborhood of  \(z\). Indeed, otherwise we have \(v'(z)\ne 0\) because \(d\) is bounded in a neighborhood of \(v(z)\), and then 
\[ \lim_{\xi \to v(z)} d(\xi)= \lim_{\xi \to v(z)} \frac{|\Phi_v(v^{-1}(\xi))|}{|v'(v^{-1}(\xi))|^{p-1}}=  \frac{|\Phi_v(z)|}{|v'(z)|^{p-1}},\]
contradicting \(v(z)\in \D(d)\) (recall that we do not consider removable discontinuities; see the Introduction).
%\[ |v'(z^+)|^{p-1}=\frac{|\Phi_v(z)|}
%{d(v(z)^-)} \ne \frac{|\Phi_v(z)|}{d(v(z)^+)}= %|v'(z^-)|^{p-1}.\]
So, \(\Phi_v\) satisfies Definition \ref{defsol}\ref{defa}.

Moreover, \(\Phi_v\) is differentiable at every $z\in I_v\setminus v^{-1}(\Theta)$, with \(\Phi_v'(z) = -\dot y(v(z)) v'(z)\). 
Then, since $y$ solves \eqref{odeprob2}, for every $z\in I_v\setminus v^{-1}(\Theta)$ we obtain
\begin{equation}
\label{eq:diffcomp}
\begin{aligned}
\Phi_v'(z)=-\dot y(v(z))v'(z) &= 
-[cg(v(z))-f(v(z))]v'(z) + \frac{d(v(z))^{\frac{1}{p-1}}h(v(z))}{y(v(z))^{\frac{1}{p-1}}}v'(z)
\\
&= -[cg(v(z))-f(v(z))]v'(z) -h(v(z)), \\
\end{aligned}
\end{equation}
i.e., \eqref{eq:differ}.

Finally, 
since \(v\) and \(\Phi_v\) are absolutely continuous in any compact interval (see Remark \ref{rem:AC}), integrating 
\eqref{eq:diffcomp} in a generic interval \([z_1,z_2]\subseteq \overline{I}_v\) produces \eqref{eq:integr} in $\overline{I}_v$. Outside $\overline{I}_v$, \eqref{eq:integr} is straightforward.

The validity of the properties in Definition \ref{defsol}\ref{defnew}--\ref{defb} is a trivial consequence of \eqref{eq:inverse} and \eqref{eq:exten}.

\end{proof}

% \cred{inizio vecchia versione}

% The next result concerns the behavior of the solutions to \eqref{odeprob2} near 0. It has been proved in \cite{M} in the case of continuous coefficients, but it is also trivially true in the present context, since \(\Theta\subset (0,1)\) and the coefficients of the equation are continuous near 0.

% \begin{lemma}{\rm (\cite[Lemma 1]{M})}\label{l:slope}   Let \(y\) be a solution of \eqref{odeprob2}.  If \eqref{diffcond} is fulfilled, then $\ell_{p}<+\infty$ and there exists \(\dot y(0^+)\). Furthermore, \(\dot y(0^+)\) is a zero of the function \(\eta(t):= |t|^{p'}-(cg(0)-f(0))|t|^{\frac{1}{p-1}}+ \ell_{p}^\frac{1}{p-1} \), where \(p'\) is the conjugate exponent of \(p\).

% \end{lemma}
% \cred{fine vecchia versione}

% \bigskip
% \cred{inizio nuova versione}

The next result provides a necessary condition for the existence of solutions. 

\begin{lemma} If problem \eqref{odeprob2} admits solutions then, put \(\lambda:=\ell_p^\frac{1}{p-1}\) (see \eqref{notation2}), 
 % \[ \dis\lambda:=\liminf_{\xi\to 0^+} \frac{\kappa(\xi)}{\xi^{\frac{1}{p-1}}},\]  
  we have \(\lambda<+\infty\). Moreover, setting \(\eta(t):= t^\frac{p}{p-1} - (cg(0)-f(0))t^\frac{1}{p-1} + \lambda\), we have  \( \dis\min_{t\ge 0}\eta(t) \le 0\).
\label{lem:nuovo}
\end{lemma}

\begin{proof}
Let \(y\) be a solution to \eqref{odeprob2}. 
Put \(\varphi(\xi):=\frac{y(\xi)}{\xi}\), \(\xi\in (0,1)\). We set
$$ L:=\limsup_{\xi\to0^+}\varphi(\xi), \quad M:=\limsup_{\xi\to0^+}\dot y(\xi). $$
Notice that $L\leq M$: indeed, taking any sequence $\xi_n\to 0$ such that $\varphi(\xi_n)\to L$, by the mean value theorem there exists $\zeta_n\to 0$ such that $\dot y(\zeta_n)=\varphi(\xi_n)$ for all $n\in\N$, yielding $M\geq L$.

We have, for all $\xi\in(0,1)\setminus\Theta$,
\beq
(\dot y(\xi)  - (cg(\xi)-f(\xi))) \varphi(\xi)^{\frac{1}{p-1}}= -\frac{\kappa(\xi)}{\xi^{\frac{1}{p-1}}}.
\label{eq:modif}
\eeq
So, if \(\lambda=+\infty\), 
 then 
\beq
\label{eq:limite}
\lim_{\xi\to 0}  (\dot y(\xi)  - (cg(\xi)-f(\xi))) \varphi(\xi)^{\frac{1}{p-1}}=-\infty.
\eeq
On the other hand, if $L=+\infty$ then $M=+\infty$ and, for a suitable sequence $\xi_n\to 0$,
$$(\dot y(\xi_n)  - (cg(\xi_n)-f(\xi_n))) \varphi(\xi_n)^{\frac{1}{p-1}}\geq 0,$$
contradicting \eqref{eq:limite}. Instead, if $L<+\infty$, then $\|\varphi\|_\infty<+\infty$ and we choose $\Lambda>\|\varphi\|_\infty^{\frac{1}{p-1}}\|cg-f\|_\infty$, so that \eqref{eq:limite} furnishes $\delta>0$ (depending on $\Lambda$) such that
$$  (\dot y(\xi)  - (cg(\xi)-f(\xi))) \varphi(\xi)^{\frac{1}{p-1}}<-\Lambda $$
for all $\xi\in(0,\delta)$, whence
$$ \dot y(\xi) \leq -\frac{\Lambda}{\|\varphi\|_\infty^{\frac{1}{p-1}}}+\|cg-f\|_\infty<0, $$
due to the choice of $\Lambda$. This forces $y$ to be negative in $(0,\delta)$, contradicting \eqref{odeprob2}.

Let us now consider the function \(\eta\) and fix a value \(\varepsilon>0\). By the definition of \(\lambda\) and \eqref{eq:modif},  there exists a positive \(\delta=\delta_\varepsilon>0\) such that
\beq(\dot y(\xi)  - (cg(\xi)-f(\xi))) \varphi(\xi)^{\frac{1}{p-1}} +\lambda \le  \varepsilon \quad \text{ for every } \xi \in (0,\delta).\label{eq:dis1}\eeq
Take any sequence $\xi_n\to0^+$ such that $\dot y(\xi_n)\to M$. Up to subsequences, $\varphi(\xi_n)\to k\in[0,+\infty]$. Then \eqref{eq:dis1} rewrites as
\beq(\dot y(\xi_n)  - (cg(\xi_n)-f(\xi_n))) \varphi(\xi_n)^{\frac{1}{p-1}} +\lambda \le  \varepsilon \quad \text{ for every } n\in\N.\label{eq:dis1seq}\eeq
If $k=+\infty$, then $L=M=+\infty$, so letting $n\to\infty$ in \eqref{eq:dis1seq} leads to a contradiction. Hence $k<+\infty$. Now, if $M=+\infty$, then $\dot y(\xi_n)>k$ for any $n$ sufficiently large, so that
$$(k  - (cg(\xi_n)-f(\xi_n))) \varphi(\xi_n)^{\frac{1}{p-1}} +\lambda \le  \varepsilon,$$
whence
$$\eta(k) = \lim_{n\to\infty} (k  - (cg(\xi_n)-f(\xi_n))) \varphi(\xi_n)^{\frac{1}{p-1}} +\lambda \leq \eps,$$
so $\displaystyle{\min_{t\geq 0} \eta(t)\leq \eta(k)\leq 0}$, since $\varepsilon$ was arbitrary. Otherwise, if both $k$ and $M$ are finite, then \eqref{eq:dis1seq} entails $(M-(cg(0)-f(0)))k^{\frac{1}{p-1}}+\lambda\leq\varepsilon$. 
Since \(k\le L\le M\), we have
$$ \eta(k) = (k-(cg(0)-f(0)))k^{\frac{1}{p-1}}+\lambda\leq (M-(cg(0)-f(0)))k^{\frac{1}{p-1}}+\lambda\leq\varepsilon $$
for all $\varepsilon>0$, yielding $\displaystyle{\min_{t\geq 0}\eta(t) \leq \eta(k)\leq 0}$. 
%Otherwise, reasoning as in the previous case, %$M<cg(0)-f(0)$ produces $%\displaystyle{\min_{t\geq 0}\eta(t) \leq \eta(M)\leq 0}$.
\end{proof}

% \cred{fine nuova versione}

We conclude this section with the following result, providing a uniform positive lower bound for the solutions of \eqref{odeprob2} in the compact subsets of \((0,1)\). 

\medskip
\begin{lemma}\label{l:lowboun}
Put \(I_r:=[r,1-r]\) for every \(r\in(0,\frac12)\). Fix \(r_0\in(0,\frac12)\) such that \(\Theta\subset I_{r_0}\).

Then, for every \(r\in (0, \frac12 r_0)\) and every \(m,M>0\), there exists \(\delta=\delta_{r,m,M}>0\) such that if 

\beq \label{ip:lowboun}
\inf_{\xi\in I_r} \kappa(\xi)\ge m \quad \text{ and } \quad
\sup_{\xi\in I_r} |cg(\xi)-f(\xi)|\le M 
\eeq
we have 
\beq
\label{eq:lowerpos}
y(\xi)\ge \delta \quad \text{ for all } \xi\in I_{2r}
\eeq
for every positive solution
\(y\in C(0,1)\cap C^1((0,1)\setminus \Theta)\) of the differential equation in \eqref{odeprob2}. \end{lemma}

\begin{proof} Fix \(r\in(0,\frac12 r_0)\) and \(m,M>0\). Choose
  $\delta=\delta_{r,m,M}<
 \min\left\{(r m)^{\frac{1}{p'}},\left(\frac{m}{p(M+1)}\right)^{p-1}\right\}\).  
Fix an arbitrary \(\xi_0\in I_{2r}\) and set $\varphi(\xi):=(\delta^{p'}-m(\xi-\xi_0))^{\frac{1}{p'}}$ for all $\xi\in [\xi_0,\xi_1]$, with \(\xi_1:=\xi_0+\frac{\delta^{p'}}{m}\). Note that \([\xi_0,\xi_1]\subset I_r\) by the choice of \(\delta\).

Assume that \eqref{ip:lowboun} holds true. Then,
for all \(\xi\in I_r\) and positive \(s<\left(\frac{m}{p(M+1)}\right)^{p-1}\), again by the choice of \(\delta\) we have

\beq
\label{subsol1}
s^{\frac{1}{p-1}}+(cg(\xi)-f(\xi))s^{\frac{1}{p-1}} - \kappa(\xi)< (M+1) s^{\frac{1}{p-1}}- m < -\frac{m}{p'}.
\eeq
Moreover,
$\varphi(\xi)\leq\delta<\left(\frac{m}{p(M+1)}\right)^{p-1}$ for every \(\xi\in [\xi_0,\xi_1]\). So, by \eqref{subsol1} we obtain
\begin{equation}
\label{eq:subsol2}
\dot \varphi(\xi)=-\frac{m}{p'\varphi(\xi)^{\frac{1}{p-1}}} >
1+cg(\xi)-f(\xi) - \frac{\kappa(\xi)}{\varphi(\xi)^{\frac{1}{p-1}}} \quad \text{ for every } \xi \in [\xi_0,\xi_1).
 \end{equation}

Moreover, since \(\varphi(\xi_1)=0<y(\xi_1)\), we can define 
\[\bar \xi:=\inf\{\xi\in[\xi_0,\xi_1]: \varphi(\tau)< y(\tau)\ \text{ for every } \tau\in [\xi,\xi_1] \}.\]
Let us prove that \(\bar\xi=\xi_0\). Indeed, if \(\bar\xi >\xi_0\), then \(\varphi(\bar\xi)=y(\bar\xi)>0\), so \eqref{eq:subsol2} entails
\begin{align*}
\dot \varphi(\bar \xi)&=\lim_{\xi\to \bar\xi^+} \dot \varphi(\xi)\ge 1+\limsup_{\xi \to \bar \xi^+} \left(cg(\xi)-f(\xi) - \frac{\kappa(\xi)}{\varphi(\xi)^{\frac{1}{p-1}}}\right) \\
 &\ge 1+\limsup_{\xi \to \bar \xi^+} \left(cg(\xi)-f(\xi) - \frac{\kappa(\xi)}{y(\xi)^{\frac{1}{p-1}}}\right)=1+ \limsup_{\xi\to\bar\xi^+}\dot y(\xi),
\end{align*}
which is a contradiction since \(\varphi(\bar\xi)=y(\bar\xi)\) and \(\varphi(\xi)<y(\xi)\) in a right neighborhood of \(\bar \xi\).
Therefore, we conclude that \(\bar\xi=\xi_0\), implying 
\(\delta= \varphi(\xi_0)\le y(\xi_0)\). 

The assertion follows from the arbitrariness of \(\xi_0\).

\begin{rmk}\label{r:AC}
As a consequence of Lemma \ref{l:lowboun}, if \(y\) is a solution of \eqref{odeprob2} then \(y\in W^{1,\infty}(I)\) for every closed interval \(I\subset (0,1)\). Therefore, since \(y\in C[0,1]\), we get that \(y\) is absolutely continuous in \([0,1]\).
Moreover, notice  that if problem \eqref{odeprob2} admits a solution then we necessarily have \(c\int_0^1 g(\xi) \dxi > \int_0^1 f(\xi) \dxi\). Indeed, if \(y\) is a solution, then
\[ 0=y(1^-)-y(0^+) = \int_0^1 \dot y(\xi) \dxi < \int_0^1 (cg(\xi)-f(\xi)) \dxi.\]
\end{rmk}

\end{proof}

\section{Regularized approximating problems}

The technique we adopt to prove the main result is
based on the regularization of the discontinuous coefficients of the equation in \eqref{odeprob}, in order to be able to use the known results concerning equations with continuous coefficients.
To this aim, let us introduce the concept of 
$\eps$-regularization of a function belonging to $\hat{C}(0,1)\cap L^1(0,1)$.

\begin{defi}
\label{epsreg}
Let $\phi\in\hat{C}(0,1)\cap L^1(0,1)$ and let \(A\supseteq \D(\phi)\) be a finite set in \((0,1)\), say \(A=\{\gamma_1,\ldots,\gamma_n\}\), being $\gamma_i<\gamma_{i+1}$ for all $i=1,\ldots,n-1$. Put $\gamma_0:=0$ and $\gamma_{n+1}:=1$. Finally, set $\overline{\eps}:=\frac{1}{2}\min\{\gamma_{i+1}-\gamma_i: \, i=0,\ldots,n\}$. 

For any $\eps\in(0,\overline{\eps})$, we call $\eps$-regularization of $\phi$ related to \(A\) the function $\phi_\eps^{(A)}:(0,1)\to\R$ defined as
\begin{equation*}
\phi_\eps^{(A)}(x)=\left\{
\begin{alignedat}{2}
&\phi(\gamma_i-\eps)+\frac{\phi(\gamma_i+\eps)-\phi(\gamma_i-\eps)}{2\eps}(x-(\gamma_i-\eps)) \quad &&\mbox{if} \;\; x\in[\gamma_i-\eps,\gamma_i+\eps], \;\; i=1,\ldots,n, \\
%&\frac{\phi(1-\eps)}{\eps}(1-x) \quad &&\mbox{if} \;\; x\in[1-\eps,1], \\
&\phi(x) \quad &&\mbox{elsewhere}. \\
\end{alignedat}
\right.
\end{equation*}
When \(A=\D(\phi)\) we simply write \(\phi_\varepsilon\).
\end{defi}

\medskip
\begin{rmk}
\label{epsregprops}

By definition, $\phi_\eps^{(A)}\in C(0,1)$ for all $\eps\in(0,\overline{\eps})$ and $\phi_\eps^{(A)}\to\phi$ locally uniformly in $(0,1)\setminus A$ as $\eps\to 0$. Moreover, it is readily seen that $(\alpha\phi)_\eps^{(A)}=\alpha \phi_\eps^{(A)}$ for all \(\alpha\in \R\) and $\eps\in(0,\overline{\eps})$.
Furthermore, if \(\phi,\psi \in \hat{C}(0,1)\cap L^1(0,1)\) and \(\alpha,\beta\in \R\), then for every finite set \(A\supseteq \D(\phi)\cup \D(\psi)\) we have \((\alpha \phi +\beta\psi)_\eps^{(A)} =  \alpha \phi_\eps^{(A)} +\beta\psi_\eps^{(A)}\).

Finally, if $I\subset(0,1)$ is a closed interval such that $\partial I \cap A = \emptyset$ and \(\phi \) is bounded in \(I\), then there exists $\hat{\eps}=\hat{\eps}(I)$ such that
\begin{equation}
\label{supest}
\sup_{I} |\phi_\eps^{(A)}| \le \sup_{I} |\phi|  \quad \mbox{for all} \;\; \eps\in(0,\hat{\eps})
\end{equation}
and
\begin{equation}
\label{infest}
\inf_I \phi_\eps^{(A)} \geq \inf_I \phi \quad \mbox{for all} \;\; \eps\in(0,\hat{\eps}).
\end{equation}
\end{rmk}

\medskip
Thanks to the \(\eps\)-regularization of the functions in \(\hat{C}(0,1)\cap L^1(0,1)\), we can reduce to problem \eqref{odeprob2} with \(f,g,h\) continuous in \([0,1]\) and \(d\) continuous in \((0,1)\) (so that \(\Theta=\emptyset\)). In such a setting, problem \eqref{odeprob2} has been investigated in \cite{M}, where an existence result has been obtained in terms of the infima and suprema of some integral averages involving the coefficients of the equation.
So, we will need  to study the behavior as \(\eps\to 0\) of such infima and suprema. In order to do that, we use  a classical tool from the Calculus of Variations: the \(\Gamma-\)convergence. Recall the following definition (see, e.g., \cite{Bra}).

\begin{defi} \label{def:Gammaconv}
Let $X$ be a topological space, $\{\Phi_n\}$ be a sequence of functions $\Phi_n:X\to[-\infty,+\infty]$, and $\Phi:X\to[-\infty,+\infty]$. We say that $\{\Phi_n\}$ $\Gamma$-converges to $\Phi$, and write $\Phi_n\stackrel{\Gamma}{\to}\Phi$, if for all $x\in X$ the following properties hold true:
\begin{enumerate}[label={(\roman*)}]
\item {\it liminf inequality:} for every $\{x_n\}\subseteq X$ such that $x_n\to x$ one has
$$ \Phi(x)\leq \liminf_{n\to\infty} \Phi_n(x_n), $$
\item {\it limsup inequality:} there exists $\{x_n\}\subseteq X$ such that $x_n\to x$ and
$$ \Phi(x)\geq \limsup_{n\to\infty} \Phi_n(x_n). $$
\end{enumerate}
\end{defi}

The following result, concerning the stability of infima under $\Gamma$-convergence, stems from \cite[Proposition 1.18(i)]{Bra}.
\begin{lemma}
\label{inf-stab}
Let $X$ be a compact topological space and $\{\Phi_n\}$ be a sequence of functions $\Phi_n:X\to[-\infty,+\infty]$ that $\Gamma$-converges to $\Phi:X\to[-\infty,+\infty]$. Then
$$ \inf_X \Phi = \lim_{n\to\infty} \inf_X \Phi_n. $$
\end{lemma}
\begin{proof}
For every $n\in\N$, choose $x_n\in X$ such that
$$ \Phi_n(x_n) < \inf_X \Phi_n + \frac{1}{n}. $$
Since $X$ is compact, then $x_n\to x\in X$, up to sub-sequences. Thus, owing to $\Phi_n\stackrel{\Gamma}{\to}\Phi$,
\begin{equation}
\label{stab1}
\inf_X \Phi \leq \Phi(x)\leq \liminf_{n\to\infty} \Phi_n(x_n) \leq \liminf_{n\to\infty} \inf_X \Phi_n.
\end{equation}
Due to the lower semi-continuity of $\Phi$ (see \cite[Proposition 1.28]{Bra}) and the compactness of $X$, there exists $y\in X$ such that $\Phi(y)=\inf_X \Phi$. Hence, according to the $\Gamma$-convergence of $\{\Phi_n\}$ to $\Phi$, there exists $\{y_n\}$ such that $y_n\to y$ and
\begin{equation}
\label{stab2}
\inf_X \Phi = \Phi(y) \geq \limsup_{n\to\infty} \Phi_n(y_n) \geq \limsup_{n\to\infty} \inf_X \Phi_n.
\end{equation}
Putting \eqref{stab1}--\eqref{stab2} together concludes the proof.
\end{proof}

\bigskip
In light of the previous result, we now prove the following technical lemma, which will allow us to get the \(\Gamma-\)convergence of the integral averages of the \(\varepsilon-\)regularizations.

\begin{lemma}
\label{lem:conv}
Let $\phi\in\hat{C}(0,1)\cap L^1(0,1)$ and let \(A\supseteq \D(\phi)\) be a finite set in \((0,1)\). Consider the $\eps$-regularization of \(\phi\) related to the set \(A\), which we simply denote by  \(\phi_\eps\), for all $\eps\in(0,\overline{\eps})$, being $\overline{\eps}$ as in Definition \ref{epsreg}. 
Assume that \(\phi\) is bounded in \(\bar I=[\bar \eps, 1-\bar\eps]\). 

For every \(x\in (0,1]\) set
\beq \label{def:Phi}
\Phi_0(x):= \fint_0^x\phi(\tau)\dtau,  \quad \Phi_\eps(x):=\fint_0^x\phi_\eps(\tau)\dtau, \quad \text{ for all } \eps\in(0,\overline{\eps}).
\eeq

Then, for every pair of sequences \(\{x_n\}\subset (0,1] \), \(\{\eps_n\}\subset[0,\bar \eps) \), with \(x_n\to \bar x\) and \(\eps_n\to 0\), we have:
\beq \label{eq:conv}
\Phi_{\eps_n}(x_n) \to \Phi_0(\bar x) \qquad \text{ if } \bar x >0
\eeq
and
\beq \label{eq:conv2}
\liminf_{x\to 0^+}\Phi_0(x)\leq \liminf_{n\to\infty}\Phi_{\eps_n}(x_n)\leq\limsup_{n\to\infty}\Phi_{\eps_n}(x_n)\leq\limsup_{x\to 0^+}\Phi_0(x).
\eeq

%In addition, there exists $\{\eps_n\}\subseteq(0,\bar\eps)$ such that, for all $\{x_n\}\subseteq(0,1]$ fulfilling $x_n\to 1$, one has 
%\beq
%\label{eq:conv3}
%\Phi_{\eps_n}(x_n)\to\Phi_0(1).
%\eeq
%\smallskip

\end{lemma}
\begin{proof}
Let us consider the function $\Psi:(0,1]\times[0,\overline{\eps})\to\R$ defined as $\Psi(x,\eps):=\Phi_\eps(x)$. If $\Psi$ is continuous at $(\bar x,0)$ for every $\bar x\in(0,1]$, then \eqref{eq:conv} holds: indeed, given any sequence $\{\eps_n\}\subset[0,\overline{\eps})$ such that $\eps\to 0$, for every $\bar x\in(0,1]$ and every sequence $x_n\to \bar x$ one has
$$ \lim_{n\to\infty} \Phi_{\eps_n}(x_n) = \lim_{n\to\infty} \Psi(x_n,\eps_n) = \Psi(\bar x,0) = \Phi_0(\bar x). $$

To this end, fix any $\bar x\in(0,1]$ and take any $\{(x_n,\eps_n)\}\subseteq (0,1]\times[0,\overline{\eps})$ such that $(x_n,\eps_n)\to(\bar x,0)$ as $n\to\infty$. %If $\bar x=0$, then $\phi_{\eps_n}\equiv \phi$ in $[0,x_n]$ for all $n\in\N$ sufficiently large, and the result is immediate. Thus we suppose $\bar x\neq 0$ and distinguish two cases: $\bar x\neq 1$ and $\bar x=1$. Let us assume $\bar x\neq 1$. 
For each $n\in\N$ we put  $\{\gamma_1,\ldots,\gamma_{m_n}\}=A\cap[0,x_n]$.

Therefore,  we have
\begin{equation*}
\Psi(x_n,\eps_n) = \frac{1}{x_n} \left[ \int_{[0,x_n]\setminus\cup_{j=1}^{m_n} B_{\eps_n}(\gamma_j)} \phi(\tau)\dtau + \sum_{j=1}^{m_n} \int_{[0,x_n]\cap B_{\eps_n}(\gamma_j)} \phi_{\eps_n}(\tau)\dtau \right],
\end{equation*}
since $B_{\eps_n}(\gamma_i)\cap B_{\eps_n}(\gamma_j)=\emptyset$ for all $i\neq j$, due to $\eps_n<\overline{\eps}$. Notice that the previous equation holds even if \(\eps_n=0\), since in this case \(B_{\eps_n}(\gamma_j)\) is empty, for every \(j=1,\ldots,m_n\).

Put \(\bar K:= \dis\sup_{\xi \in \bar I}|\phi(\xi)|\), for any $j\in\{1,\ldots,m_n\}$ we have
\begin{equation*}
%\label{fundamental}
\begin{aligned}
\int_{[0,x_n]\cap B_{\eps_n}(\gamma_j)} |\phi_{\eps_n}(\tau)| \dtau &\leq \int_{\gamma_j-\eps_n}^{\gamma_j+\eps_n} \left|\phi_{\eps_n}(\tau)\right| \dtau \\
&= \int_{0}^{2\eps_n} \left| \left[\phi(\gamma_j-\eps_n)+\frac{\phi(\gamma_j+\eps_n)-\phi(\gamma_j-\eps_n)}{2\eps_n}t\right] \right| \dt \\
&\leq \eps_n \left[2|\phi(\gamma_j-\eps_n)| + |\phi(\gamma_j+\eps_n)-\phi(\gamma_j-\eps_n)|\right] \leq 4\bar K  \eps_n.
\end{aligned}
\end{equation*}
Accordingly,
\begin{equation*}
\left| \sum_{j=1}^{m_n} \int_{[0,x_n]\cap B_{\eps_n}(\gamma_j)} \phi_{\eps_n}(\tau)\dtau \right| \leq \sum_{j=1}^{m_n} \int_{[0,x_n]\cap B_{\eps_n}(\gamma_j)} \left|\phi_{\eps_n}(\tau)\right| \dtau \leq 4\sharp(A)\bar K \eps_n   \to 0 \quad \text{as } n\to\infty,
\end{equation*}
 where $\sharp(A)$ denotes the cardinality of $A$. Moreover, 
$$\int_{[0,x_n]\setminus\cup_{j=1}^{m_n} B_{\eps_n}(\gamma_j)} \phi(\tau)\dtau \to \int_{[0,\bar x]} \phi(\tau)\dtau \quad \mbox{as} \;\; n\to\infty. $$
We conclude that $\Psi(x_n,\eps_n) \to \Psi(\bar x,0)$ as $n\to\infty$.

Finally, \eqref{eq:conv2} is an immediate consequence of the fact that $\phi_{\eps_n}\equiv \phi$ in $[0,x_n]$ for all $n\in\N$ sufficiently large.

\end{proof}

\bigskip
As a consequence of the previous lemma, we have the following convergence result.

\begin{prop} Under the same assumptions of Lemma \ref{lem:conv}, assume also that \[\dis\liminf_{x\to 0^+} \fint_0^x \phi(\tau) \dtau<+\infty.\]

Then, 
we have
\beq \label{eq:phiconv}
\inf_{x\in (0,1)} \fint_0^x \phi(\tau) \dtau = \lim_{n\to \infty} \inf_{x\in (0,1)} \fint_0^x\phi_{\eps_n} (\tau)\dtau.
\eeq

\noindent
Similarly, if 
\[\dis\limsup_{x\to 0^+} \fint_0^x \phi(\tau) \dtau<+\infty\]
we have
\beq \label{eq:phiconv2}
\sup_{x\in (0,1)} \fint_0^x \phi(\tau) \dtau = \lim_{n\to \infty} \sup_{x\in (0,1)} \fint_0^x\phi_{\eps_n} (\tau)\dtau.
\eeq

\label{pr:Gammac}

\end{prop}

\begin{proof}
As in Lemma \ref{lem:conv}, let us define \(\Phi_0(x)\) and \(\Phi_\eps(x)\) for \(x\in (0,1]\) by \eqref{def:Phi}. Moreover,  put
\beq \Phi_0(0)=\Phi_\eps(0)=\liminf_{x\to 0^+} \Phi_0(x).\label{eq:phi0} \eeq
Fix any sequence \(\{\eps_n\}\subset [0,\bar\eps)\) convergent to 0.

\smallskip
Since \(\phi_\eps\equiv \phi\) in a right neighborhood of 0, we have also \(\Phi_\eps\equiv \Phi_0\) in the same neighborhood. So, 
by \eqref{eq:conv2}, for any sequence \(x_n\to 0\) we have 
\[ \Phi_0(0) \le \liminf_{n \to +\infty} \Phi_0(x_n) = \liminf_{n \to +\infty} \Phi_{\eps_n}(x_n).\]
Moreover, if \(\{\bar x_n\}\) is a sequence such that \(\Phi_0(\bar x_n)\to \Phi_0(0)\), then we get 
also \(\Phi_0(0)\ge \dis\limsup_{n\to +\infty} \Phi_{\eps_n}(\bar x_n)\). 

Therefore, taking also \eqref{eq:conv} into account, we conclude that \(\Phi_{\eps_n} \stackrel{\Gamma}\to \Phi_0 \) in \([0,1]\), so 
from Lemma \ref{inf-stab} we get
\[ \inf_{x\in (0,1)} \fint_0^x \phi(\tau) \dtau =\inf_{x\in [0,1]} \Phi_0(x) = \lim_{n\to \infty} \inf_{x\in [0,1]} \Phi_{\eps_n}(x) = \lim_{n\to \infty} \inf_{x\in (0,1)} \fint_0^x\phi_{\eps_n} (\tau)\dtau,\]
that is, 
\eqref{eq:phiconv}.

The second part is analogous. Indeed, setting
\[ \Phi_0(0)=\Phi_\eps(0)=\limsup_{x\to 0^+} \Phi_0(x)\]
instead of \eqref{eq:phi0}, we have that   \(-\Phi_{\eps_n} \stackrel{\Gamma}\to -\Phi_0 \), from which we deduce the validity of \eqref{eq:phiconv2}.

\end{proof}

\section{Main results}

Now, we have all the tools to achieve the existence result for the discontinuous equation by means of  that concerning the regularized one. Using a result contained in \cite{M}, we prove the following existence result pertaining problems with continuous coefficients.

\begin{prop}
\label{p:exten}
Let \(f,g,\kappa\) be continuous functions defined in \([0,1]\), with \(\kappa(0)=\kappa(1)=0\) and \(\kappa(\xi)>0\) in \((0,1)\). Let \(K_0\) be defined by \eqref{notation}. If
\begin{equation*}
\inf_{\xi\in (0,1)} \fint_0^\xi (cg(\tau)-f(\tau)) \dtau > p'(p-1)^\frac{1}{p} K_0^\frac{1}{p'},
%\label{eq:stima forte}
\end{equation*}
then problem \eqref{odeprob2} admits a solution. 
\end{prop}

\begin{proof}
Put \(\beta:=\dis\inf_{\xi\in(0,1)} \fint_0^\xi (cg(\tau)-f(\tau)) \dtau  \), the function \(M(t):= t^{\frac{p}{p-1}}-\beta t^{\frac{1}{p-1}} +K_0\), \(t\ge 0\), admits minimum attained at \(\lambda:=\beta/p\) with \[M(\lambda)=
K_0-\frac{1}{p'}\left(\frac{\beta^p}{p}\right)^\frac{1}{p-1}<0.\]
Accordingly, we have
\(\lambda^\frac{p}{p-1} <\beta \lambda^\frac{1}{p-1}-K_0 \), implying \[\lambda<\beta - \frac{K_0}{\lambda^\frac{1}{p-1}} \le \fint_0^\xi (cg(\tau)-f(\tau)) \dtau - \frac{1}{\lambda^\frac{1}{p-1}}\fint_0^\xi \frac{\kappa(\tau)}{\tau^{\frac{1}{p-1}}}\dtau = \fint_0^\xi (cg(\tau)-f(\tau)) \dtau - \fint_0^\xi \frac{\kappa(\tau)}{(\lambda\tau)^{\frac{1}{p-1}}}\dtau\]
for all $\xi\in(0,1)$. Therefore, put \(\psi(\xi):=\lambda\xi\), we have
\[ \psi(\xi)< \int_0^\xi \left( cg(\tau) -f(\tau) -   \frac{\kappa(\tau)}{\psi(\tau)^{\frac{1}{p-1}}}\right)\dtau,\]
that is, \(\psi\) is an integral lower-solution of problem \eqref{odeprob2}.
By \cite[Corollary 1]{M} we conclude that problem \eqref{odeprob2} admits a solution \(y\in C^1(0,1)\).
    
\end{proof}

%\begin{rmk}
%Notice that the existence result given by Proposition \ref{p:exten} holds true even without any assumptions about the sign of \(g(0)\neq 0\) and of the integral function of $g$. As it was showed in \cite{M}, such requirements serve to show that the set $\C$ of the admissible speeds (see Definition \ref{def:Gamma-c}) is connected and in order to obtain the estimate \eqref{threshold} for the threshold speed \(c^*\).
%\end{rmk}

\bigskip

Now we are able to prove our main results concerning the existence of t.w.s.'s.

\begin{thm}
\label{t:mainthm1}  
According to the notations \eqref{notation2}, we have:
\begin{enumerate}[label={{\rm (\arabic*)}}]
\item \label{thm1case1} if $\ell_{p}=+\infty$, then there exist no t.w.s.'s for any $c\in\R$.
\item \label{thm1case3} if $\ell_{p}<+\infty$ and
\begin{equation}
\label{nonexcond}
cg(0)-f(0) < p'[\ell_p(p-1)]^{\frac{1}{p}},
\end{equation}
then there exist no t.w.s.'s having speed $c$;

\item \label{thm1case2} if $L_{p}<+\infty$ and
\begin{equation}
\label{excond}
\inf_{\xi\in(0,1)} \fint_0^\xi (cg(\tau)-f(\tau)) \dtau > p'(p-1)^{\frac{1}{p}}\left(\sup_{\xi\in(0,1)} \fint_0^\xi \frac{\kappa(\tau)}{\tau^{\frac{1}{p-1}}} \dtau\right)^{\frac{1}{p'}},
\end{equation}
where \(\kappa\) is defined by \eqref{ip:kappa}, then there exists a t.w.s. having speed $c$ and it is unique, up to shifts.

\end{enumerate}

\end{thm}

\begin{proof}
The conclusion in case \ref{thm1case1} follows directly from Proposition \ref{prop:equiv} and Lemma \ref{lem:nuovo}. As for 
 the case \ref{thm1case3}, simple calculations show that assumption \eqref{nonexcond} ensures that the function \(\eta(t)\) involved in Lemma \ref{lem:nuovo} is positive for every $t\geq 0$. So, no solutions to problem \eqref{odeprob2} exist.

So, from now on consider 
the case \ref{thm1case2}.  

Let us fix \(c \in \R\) fulfilling \eqref{excond} and set
\[\psi(\tau):= \frac{\kappa(\tau)}{\tau^{\frac{1}{p-1}}}, \quad \tau\in (0,1).\]
By \eqref{ip:hatC} we have \(\psi\in \hat C(0,1)\); moreover, since $L_p<+\infty$, by \eqref{ip:kappa} we get \(\psi\in L^1(0,1)\).
\newline
Moreover, put \(A:=\D(g)\cup \D(f)\), 
let us consider the \(\eps-\)regularizations \(g_\eps^{(A)}\), \(f_\eps^{(A)}\) and \(\psi_\eps\) obtained as in Definition \ref{epsreg}. Set 
$H_\eps(\xi):=cg_\eps^{(A)}(\xi)-f_\eps^{(A)}(\xi)$, for all $\xi\in[0,1]$.
By Remark \ref{epsregprops} we get \(H_\eps(\xi)= (cg(\xi)-f(\xi))_\eps^{(A)}\), as well as
\begin{equation}
\label{locunifconv1}
H_\eps(\xi)\to cg(\xi)-f(\xi) \quad \mbox{and} \quad \psi_\eps\to \psi\quad \mbox{locally uniformly in} \;\; (0,1)\setminus\Theta \quad \text{as } \eps \to 0.
\end{equation}
Furthermore, since \(\kappa\) is bounded in every closed  \(I\subset (0,1)\), the same holds true for \(\psi\). Hence, since $L_p<+\infty$, by Proposition \ref{pr:Gammac} we have

% Notice that $H\in \hat{C}[0,1]\cap L^\infty[0,1]$ by \eqref{ip:hatC}, while \(\kappa\in \hat C(0,1)\), by \eqref{ip:d} and \eqref{ip:kappa}. 
% Moreover, since $L_p<+\infty$, also the map $\psi(\tau):= \frac{\kappa(\tau)}{\tau^{\frac{1}{p-1}}}$ belongs to $\hat C(0,1)\cap L^1(0,1)$.

 %Let us consider the \(\eps-\)regularizations \(H_\eps\) and \(\psi_\eps\) obtained as in Definition \ref{epsreg}. Since \(\D(H)\cup \D(\psi)\subseteq \Theta\), by Remark \ref{epsregprops} we get 

%Since \(L_p<+\infty\), we can apply Proposition \ref{pr:Gammac}, and obtain
\begin{equation*}
%\label{infsup-conv}
\begin{aligned}
\lim_{\eps\to0} \inf_{\xi\in(0,1)} \fint_0^\xi H_\eps(\tau)\dtau &= \inf_{\xi\in(0,1)} \fint_0^\xi (cg(\tau)- f(\tau))\dtau, \\
\lim_{\eps\to0} \sup_{\xi\in(0,1)} \fint_0^\xi \psi_\eps(\tau)\dtau &= \sup_{\xi\in(0,1)} \fint_0^\xi \psi(\tau)\dtau.
\end{aligned}
\end{equation*}
%Taking into account that $\xi\mapsto\fint_0^\xi H(\tau)\dtau$ and $\xi\mapsto\fint_0^\xi\frac{\kappa(\tau)}{\tau^{\frac{1}{p-1}}}\dtau$ are continuous,
Then, by \eqref{excond}, there exists \(\eps^*>0\) such that
\beq \inf_{\xi\in(0,1)} \fint_0^\xi H_\eps(\tau)\dtau > p'(p-1)^{\frac{1}{p}}\left(\sup_{\xi\in(0,1)} \fint_0^\xi \psi_\eps(\tau)\dtau\right)^{\frac{1}{p'}} \quad \text{ for all } \eps\in(0,\eps^*). \label{eq:interm}\eeq
Let us define, for every $\eps\in(0,\eps^*)$,
\[ \tilde \psi_\eps(x):=\begin{cases} 
\min\{\psi_\eps(x), \frac{\psi_\eps(\eps)}{\eps}x\} & \text{ for } 0\le x<\eps, \\
\psi_\eps(x) & \text{ for } \eps\le x\le 1-\eps,  \\   \min\{\psi_\eps(x), \frac{\psi_\eps(1-\eps)}{\eps}(1-x)\} & \text{ for } 1-\eps<x\le 1.  \end{cases}\]
Of course, \(\tilde \psi_\eps\) is continuous in \([0,1]\) with \(\tilde\psi_\eps(0)=\tilde\psi_\eps(1)=0\).
Since \(\tilde \psi_\eps(x)\le \psi_\eps(x)\) for every \(x\in [0,1]\), by \eqref{eq:interm} we have 

\beq \inf_{\xi\in(0,1)} \fint_0^\xi H_\eps(\tau)\dtau > p'(p-1)^{\frac{1}{p}}\left(\sup_{\xi\in(0,1)} \fint_0^\xi \tilde\psi_\eps(\tau)\dtau\right)^{\frac{1}{p'}} \quad \text{ for all } \eps\in(0,\eps^*). \label{eq:final}\eeq

Put now 
\[ \tilde \kappa_\eps(\xi):= \tilde \psi_\eps(\xi) \xi^\frac{1}{p-1}.\]
Of course, \(\tilde \kappa_\eps\) is continuous in \([0,1]\) and, by \eqref{locunifconv1}, \(\tilde \kappa_\eps \to \kappa\)
locally uniformly in \((0,1)\setminus \Theta\) as \(\eps\to 0\).
Therefore, by \eqref{eq:final},
as a consequence of Proposition \ref{p:exten}, for any $\eps\in(0,\eps^*)$ there exists a solution $y_\eps\in C^1(0,1)$ to
\begin{equation}
\label{eq:regodeprob}
\left\{
\begin{alignedat}{2}
\dot y &= cg_\eps^{(A)}(\xi)-f_\eps^{(A)}(\xi)-\frac{\tilde\kappa_\eps(\xi)}{y^{\frac{1}{p-1}}} \quad &&\mbox{in} \;\; (0,1), \\
y&>0 \quad &&\mbox{in} \;\; (0,1), \\
y(0)&=y(1)=0.
\end{alignedat}
\right.
\end{equation}

Let us now prove that the solutions \(y_\eps\) are equi-continuous in any compact interval $I\subset (0,1)$. To this aim, for every \(r\in(0, \frac12) \) put \(I_r:=[r,1-r]\) and let \(r_0<\frac12\) be such that \(\Theta\subset I_{r_0}\).

Let us fix \(r<\frac12 r_0\) and 
put \(M:=\|cg-f\|_\infty\) and \(m_r:=r^{\frac{1}{p-1}}\inf_{I_{r}} \psi\).
By \eqref{supest} we have \(\|H_\eps\|_\infty\le M\), while \eqref{infest} entails \(\inf_{I_r} \psi_\eps\ge \inf_{I_{r}} \psi\).
So, \(\inf_{I_{r}} \tilde\kappa_\eps \ge m_r\).  Hence, 
we can apply Lemma \ref{l:lowboun} to deduce that 
there exists a positive \(\delta_{r}\) such that 
\begin{equation}
\label{eq:lowerbound}
y_\eps(\xi)\geq \delta_{r} \quad \mbox{for all} \;\; \xi\in I_{2r} \;\; \mbox{and} \;\; \eps\in(0,\eps^*).
\end{equation}
Moreover, notice that \(\psi\in L^\infty (I_{2r})\), and
\begin{equation*}
\tilde \kappa_\eps(\xi) < \tilde \psi_\eps(\xi)\le \psi_\eps(\xi) \le \|\psi\|_{L^\infty(I_{2r})} \quad \text{ for every } \xi \in I_{2r}.
%\label{eq:kappa}
\end{equation*}
Thus, 
\begin{equation}
\label{derest}
 %\int_{I_{2r}} |\dot y_\eps(\tau)| \dtau \leq \|%H_\eps\|_\infty+\frac{1}{\delta_r^{\frac{1}
 %{p-1}}}\int_0^1 \tilde \kappa_\eps(\tau) \dtau \leq %M+\frac{\|\psi\|_{L^\infty(I_{2r})} }%{\delta_r^{\frac{1}{p-1}}}.
 |\dot y_\eps(\xi)|\le M+ \frac{\|\psi\|_{L^\infty(I_{2r})} }{\delta_r^{\frac{1}{p-1}}} \quad \text{ for every } \xi \in I_{2r},
\end{equation}
implying that  $\{y_\eps\}$ is equi-continuous in $I_{2r}$ (see Remark \ref{r:AC}).

Moreover,
notice that 
\beq  \label{eq:equiboun} 0\le y_\eps(\xi) = \int_0^\xi \dot y_\eps(\tau) \dtau \le \int_0^\xi H_\eps(\tau)\dtau \le \|H_\eps\|_\infty \xi \le M \quad \mbox{ for any } \xi  \in (0,1),  \eeq
for every \( \eps<\bar \eps\), so $\{y_\eps\}$ is also equi-bounded in $I_{2r}$.
Therefore, up to sub-sequences, Ascoli-Arzelà's Theorem ensures $y_\eps\to \hat y_r$ uniformly in $I_{2r}$ for some $\hat y_r\in C(I_{2r})$. By using a diagonal argument we achieve the existence of $y\in C(0,1)$ such that, up to a subsequence,
\begin{equation}
\label{eq:locunifconv2}
y_\eps\to y \quad \mbox{locally uniformly in} \;\; (0,1).
\end{equation}

Finally, according to \eqref{eq:regodeprob}, \eqref{eq:lowerbound}, and
\eqref{eq:locunifconv2}, 
   for every $\xi_0,\xi\in (0,1)$ we have 
\begin{equation*}
\begin{array}{ll}y(\xi)-y(\xi_0)& =\dis \lim_{\eps\to 0}\left[y_\eps(\xi)-y_\eps(\xi_0)\right] = \dis\lim_{\eps\to 0} \int_{\xi_0}^\xi \left[H_\eps(\tau)-\frac{\tilde\kappa_\eps(\tau)}{y_\eps(\tau)^{\frac{1}{p-1}}}\right] \dtau  \\
  & = \dis \int_{\xi_0}^\xi \left[cg(\tau)-f(\tau)-\frac{\kappa(\tau)}{y(\tau)^{\frac{1}{p-1}}}\right] \dtau. \end{array}
\end{equation*}
Hence, we deduce that $y\in C^1((0,1)\setminus\Theta)$ and it is a solution of the differential equation in \eqref{odeprob2}. Moreover, \eqref{eq:lowerbound} guarantees that \(y(\xi)>0\) in \((0,1)\).

\medskip
As for the boundary conditions, note that by  \eqref{eq:equiboun} we have \(0\le y(\xi)\le M \xi\) for every \(\xi \in (0,1)\), whence \(y(0^+)=0\). Moreover, again by \eqref{eq:equiboun}, for every \(\xi \) sufficiently close to 1 we have
%\[\]
% y(0^+) =
%\lim_{\xi\to 0^+} y(\xi) = \lim_{\xi\to 0^+} %\lim_{\eps\to 0^+} y_\eps(\xi) = \lim_{\eps\to 0^+} \lim_{\xi\to 0^+} y_\eps(\xi) = 0. $$
%\[\left|\frac{\rm d}{{\rm d}\xi} (y_\eps(\xi)^{p'})\right|= p'\left|H_\eps(\xi)y_\eps(\xi)^{\frac{1}{p-1}}-\tilde\kappa_\eps(\xi)\right|\leq p'(\|H\|_\infty^{p'} + \|K\|_\infty)=:M \quad \mbox{ for any } \xi\in (0,1), \eps<\bar\eps.\]
%\le M(1-\xi) \quad \mbox{ for any } \xi\in %(0,1), \ \eps<\bar\eps.  \]
\[\begin{alignedat}{2}  y_\eps(\xi)^{p'} & = \dis \int_\xi^1 \left[-\frac{\rm d}{{\rm d}\tau} (y_\eps(\tau)^{p'})\right] \dtau 
=p' \int_\xi^1 [\tilde \kappa_\eps (\tau) - H_\eps(\tau)y_\eps(\tau)^\frac{1}{p-1}] \dtau \le p'\int_\xi^1 \tilde \kappa_\eps (\tau)\dtau \\
& \le p' \dis\int_\xi^1 \psi_\eps(\tau)\tau^{\frac{1}{p-1}}\dtau\le p'\int_\xi^1 \psi_\eps(\tau)\dtau=p'\int_\xi^1 \psi(\tau)\dtau.
\end{alignedat}\]
Thus, since \(\psi\in L^1(0,1)\),  we  deduce \(y(1^-)=0\). Therefore, $y$ solves \eqref{odeprob2}. 

\medskip
Finally, assume by contradiction the existence of two solutions \(y_1, y_2\) of problem \eqref{odeprob2} and assume, w.r., \(y_1(\xi_0)<y_2(\xi_0)\) for some \(\xi_0\in (0,1)\). Let \((a,b)\) be the maximal interval in \((0,1)\), containing \(\xi_0\), where \(y_1(\xi)<y_2(\xi)\). By the boundary conditions in \eqref{odeprob2}, we deduce that \(y_1(a^+)=y_2(a^+)\) and \(y_1(b^-)=y_2(b^-)\). But

%From the uniqueness of the solution of the differential equation in \eqref{odeprob2} passing through a given point, we deduce \(y_1(\xi)\ne y_2(\xi)\) for every \(\xi\in (0,1)\). So, we can assume w.r. \(y_1(\xi)<y_2(\xi)\) in \((0,1)\). Then, 
\[ \begin{alignedat}{2}
    y_1(b^-)-y_1(a^+)&=\int_a^b \left[ cg(\xi)-f(\xi) - \frac{\kappa(\xi)}{y_1(\xi)^{\frac{1}{p-1}}}\right] \dxi 
    <\int_a^b \left[cg(\xi)-f(\xi) - \frac{\kappa(\xi)}{y_2(\xi)^{\frac{1}{p-1}}}\right] \dxi \\ &= y_2(b^-)-y_2(a^+), \end{alignedat}\]
a contradiction.

\end{proof}

The next result concerns the structure of the set of the admissible wave speeds.
% Let us define
% \beq \label{def:Gamma-c}
% \C:=\{c\in\R: \, \mbox{there exists a t.w.s. of speed } \, c\}.
% \eeq

\begin{thm}
\label{t:mainthm2}
Suppose that $g(0)>0$ and \(L_p<+\infty\).
Then, if \eqref{excond} is satisfied for some \(c\in \R\), the set $\C$ (defined in \eqref{def:Gamma-c}) is non-empty and bounded from below. Moreover, \(\C\) admits minimum $c^*$, which fulfills
\begin{equation}
\label{eq:c*cond1}
\inf_{\xi\in(0,1)} \fint_0^\xi (c^*g(\tau)-f(\tau)) \dtau \leq p'(p-1)^{\frac{1}{p}}\left(\sup_{\xi\in(0,1)} \fint_0^\xi \frac{\kappa(\tau)}{\tau^{\frac{1}{p-1}}}\dtau\right)^{\frac{1}{p'}}
\end{equation}
and
\beq \label{eq:c*cond2}
c^* g(0)-f(0) \geq p'(p-1)^{\frac{1}{p}}\ell_p^{\frac{1}{p}}.
\eeq
%where \(H_c^*(\xi):=c^*g(\xi)-f(\xi)\).

Finally, if $\int_0^\xi g(\tau)\dtau\ge 0$  in $[0,1]$, then $\C=[c^*,+\infty)$.
\end{thm}

\begin{proof}
Since $g(0)>0$, there exists $\hat{c}\in\R$ such that \eqref{nonexcond} is satisfied for all $c<\hat{c}$, so $\C$ is bounded from below by Theorem \ref{t:mainthm1}. Set
\(c^*:=\inf \C\).

First of all, note that \(c^*\) satisfies \eqref{eq:c*cond2}. Indeed, if not, then we would have \eqref{nonexcond} for every \(c\) in a right neighborhood of \(c^*\), implying by Theorem \ref{t:mainthm1} that no t.w.s. having speed \(c\) exist, in contrast with the definition of \(c^*\).

As regards the validity of \eqref{eq:c*cond1}, assume, by contradiction, that for some 
$\eps>0$ we have

\[\inf_{\xi\in(0,1)} \fint_0^\xi (c^*g(\tau)-f(\tau)) \dtau > p'(p-1)^{\frac{1}{p}}\left(\sup_{\xi\in(0,1)} \fint_0^\xi \frac{\kappa(\tau)}{\tau^{\frac{1}{p-1}}} \dtau\right)^{\frac{1}{p'}}+\eps.\]
Then, 
\[\fint_0^\xi (c^*g(\tau)-f(\tau)) \dtau > p'(p-1)^{\frac{1}{p}}\left(\sup_{\xi\in(0,1)} \fint_0^\xi \frac{\kappa(\tau)}{\tau^{\frac{1}{p-1}}} \dtau\right)^{\frac{1}{p'}}+\eps, \quad \text{ for every } \xi\in (0,1).\]
Put \(L:=\dis\sup_{\xi\in (0,1)}\fint_0^\xi g(\tau) \dtau>0\) and set \(\delta:=\eps/L\). Then, for every \(\xi\in (0,1)\) we have 
\[
\begin{aligned}\fint_0^\xi ((c^*-\delta)g(\tau)-f(\tau)) \dtau &= \fint_0^\xi (c^*g(\tau)-f(\tau)) \dtau - \delta\fint_0^\xi g(\tau)\dtau \\ 
&\geq  \fint_0^\xi (c^*g(\tau)-f(\tau)) \dtau - \eps >
p'(p-1)^{\frac{1}{p}}\left(\sup_{\xi\in(0,1)} \fint_0^\xi \frac{\kappa(\tau)}{\tau^{\frac{1}{p-1}}} \dtau\right)^{\frac{1}{p'}}. 
\end{aligned}\]
Hence condition \eqref{excond} is fulfilled for \(c=c^*-\delta\), in contrast with the definition of \(c^*\). So, \eqref{eq:c*cond1} holds true.

\smallskip Let us now prove that \(c^*\in \C\). To this aim, 
let $\{c_n\}\subseteq\C$ be a decreasing sequence converging to $ c^*$ and consider a sequence $\{y_n\}$ of solutions to \eqref{odeprob2} with $c=c_n$, which are given by Theorem \ref{t:mainthm1} and Proposition \ref{prop:equiv}.
Set \(\hat c:=\max |c_n|\). Fix any \(r\in(0,\frac14)\) such that \(\Theta \subset (2r,1-2r)\), and put \(M:=\|f\|_\infty +\hat c\|g\|_\infty\). By virtue of Lemma \ref{l:lowboun} we deduce that there exists a positive
 \(\delta_{r}\) such that 
\beq \label{eq:lowerbound2}
y_n(\xi) \ge \delta_{r} \quad \text{ for every } \xi \in I_{2r}:=[2r,1-2r] \text{ and } n\in \N.
\eeq  
Therefore, we have
\begin{equation*}
|\dot y_n(\xi)|\le \hat c\|g\|_\infty+ \|f\|_\infty+ \frac{1}{\delta_r^{\frac{1}{p-1}}}\int_0^1 \kappa(\tau)\dtau \quad \text{ for every } \xi\in I_{2r}.
%\label{derest2}
\end{equation*}

%for any $\xi_1,\xi_2\in I_r$ we have
%\begin{equation}
%\label{derest2}
%|y_n(\xi_1)-y_n(\xi_2)| \leq \int_{I_r} |\dot y_n(\tau)|\dtau \leq \hat c\|g\|_\infty+ \|f\|_\infty+ \frac{1}{\delta_r^{\frac{1}{p-1}}}\int_0^1 \kappa(\tau)\dtau.
%\end{equation}
Thus, reasoning as for \eqref{derest} and \eqref{eq:equiboun}, we get the compactness of $\{y_n\}$ in $C(I_{2r})$ via Ascoli-Arzelà's theorem. Then $y_n\to \hat y_r$ uniformly in $I_{2r}$ for some $\hat y_r\in C(I_{2r})$. Finally, a diagonal argument provides a function $y^*\in C(0,1)$ such that, up to a subsequence, $y_n\to y^*$ locally uniformly in $(0,1)$, with   \(y^*(\xi)>0\) in \((0,1)\) from \eqref{eq:lowerbound2}. 
So, for each fixed \(\xi_0,\xi\in (0,1)\), the sequence \(\{y_n\}\) uniformly converges to \(y^*\) in \([\xi_0,\xi]\), implying 
\[
\begin{aligned}
y^*(\xi)-y^*(\xi_0) &= \lim_{n\to\infty} \left[y_n(\xi)-y_n(\xi_0)\right] = \lim_{n\to\infty} \int_{\xi_0}^{\xi} \left[c_n g(\tau)- f(\tau) - \frac{\kappa(\tau)}{y_n(\tau)^{\frac{1}{p-1}}}\right] \dtau \\
&= \int_{\xi_0}^\xi \left[c^*g(\tau)-f(\tau)-\frac{\kappa(\tau)}{y^*(\tau)^{\frac{1}{p-1}}}\right] \dtau.
\end{aligned}
\]
Therefore, \(y^*\in C^1((0,1)\setminus \Theta)\) and is a solution of the differential equation in \eqref{odeprob2}. Moreover,
arguing as in Theorem \ref{t:mainthm1}, we can prove that \(y^*(0^+)=y^*(1^-)=0\), so that $y^*$  solves \eqref{odeprob2} with $c=c^*$, implying that \(c^*\in \C\).

\smallskip
Finally, let us assume $\int_0^\xi g(\tau)\dtau\geq 0$ for all $\xi\in[0,1]$ and prove that \(\C=[c^*,+\infty)\). To this aim, fix a speed $\bar c>c^*$.

Observe that for every \(\xi \in (0,1)\) we have
\[  y_{c^*}(\xi)=\int_0^\xi \left[c^*g(\tau)-f(\tau)-\frac{\kappa(\tau)}{y_{c^*}(\tau)^{\frac{1}{p-1}}}\right]\dtau \leq \int_0^\xi \left[\bar c g(\tau)-f(\tau)-\frac{\kappa(\tau)}{y_{c^*}(\tau)^{\frac{1}{p-1}}}\right]\dtau=:\gamma(\xi).\]
So, 
\begin{equation*}
%\label{eq:ycsotto}
\dot\gamma(\xi) = \bar c g(\xi)-f(\xi)-\frac{\kappa(\xi)}{y_{c^*}(\xi)^{\frac{1}{p-1}}} \le  \bar c g(\xi)-f(\xi)-\frac{\kappa(\xi)}{\gamma(\xi)^{\frac{1}{p-1}}} \quad \text{ for every } \xi\in (0,1)\setminus \Theta.
\end{equation*}

%Note that \(\gamma\) is not a solution of the equation in \eqref{odeprob2} for \(c=c^*\), in any interval \(J\subset (0,1)\). In fact, otherwise, by \eqref{eq:ycsotto} we would have \(y_c^*(\xi)=\gamma(\xi)\) in such an interval, implying that \(y_{c^*}\) is a solution of differential equation also for \(c=\bar c\), a contradiction since 

Take an increasing sequence \(\{\xi_n\}\) converging to \(1\), with \(\xi_1>\max \Theta\), and  consider, for each \(n \in \N\), the (unique) solution \(y_n\) of the equation in \eqref{odeprob2} with \(c=\bar c\), passing through \((\xi_n , \gamma(\xi_n))\), defined in its maximal existence interval \((\alpha_n, \beta_n)\).

We have \(y_n(\xi)\le \gamma(\xi)\) in \((\alpha_n, \xi_n)\)
  and \(y_n(\xi)\ge \gamma(\xi)\) in \((\xi_n,\beta_n)\).
 Indeed, if  \(y_n(\xi^*)> \gamma(\xi^*)\) for some \(n\in \N\) and \(\xi^*\in (\alpha_n,\xi_n)\) then, put \[\bar \xi:= \sup\{\xi>\xi^*: y_n(\tau)>\gamma(\tau) \text{ for every } \tau\in [\xi^*,\xi]\},\] we have 
 \(y_n(\bar \xi)=\gamma(\bar \xi)\) and
 \begin{equation*}
\begin{alignedat}{2}
 y_n(\bar \xi) - y_n(\xi^*)&=\int_{\xi^*}^{\bar \xi} \dot y_n(\tau)\dtau = \int_{\xi^*}^{\bar \xi} \left(\bar c g(\tau)-f(\tau) - \frac{\kappa(\tau)}{y_n(\tau)^{\frac{1}{p-1}}}\right) \dtau \\
 &> \int_{\xi^*}^{\bar \xi} \left(\bar c g(\tau)-f(\tau) - \frac{\kappa(\tau)}{\gamma(\tau)^{\frac{1}{p-1}}}\right) \dtau \ge \int_{\xi^*}^{\bar \xi} \dot \gamma(\tau)\dtau=\gamma(\bar \xi)-\gamma(\xi^*),    
\end{alignedat}
\end{equation*}
so \(y_n(\xi^*)<\gamma(\xi^*)\), a contradiction. The proof of the inequality in \((\xi_n,\beta_n)\) is analogous.

Therefore, 
\(\alpha_n=0\) and \(y_n(0^+)=0\) for each \(n\): indeed, if \(\alpha_n>0\) for some \(n\) then $y(\alpha_n^+)=0$ and, by the equation in \eqref{odeprob2}, we derive \(\dot y_n(\alpha_n^+)=-\infty\),  a contradiction.
Moreover,  \(\beta_n=1\) and \(y_n(1^-)\ge \gamma(1)\), for every \(n\in \N\). Indeed, if \(\beta_n<1\) for some \(n\) then by the equation in \eqref{odeprob2} we derive \(\dot y_n(\beta_n^-)\leq cg(\beta_n)-f(\beta_n)<+\infty\), with \(y(\beta_n^+)=+\infty\),  a contradiction.

From  Lemma \ref{l:lowboun} we deduce that for every 
\(r\) sufficiently small  there exists \(\delta_r>0 \) such that \(y_n(\xi)\ge \delta_r\) for every \(\xi\in I_{2r}\) and every \(n\in \N\) (see \eqref{eq:lowerpos}). Hence, reasoning as above, the sequence \(\{y_n\}\) is equi-continuous and equi-bounded in each compact interval in \((0,1)\). So, for each \(I\subset (0,1)\) we find a subsequence of \(\{y_n\}\) uniformly convergent in \(I\) to a solution \(y_{\bar c}\) of the equation in \eqref{odeprob2}, and a diagonal argument allows us to extend the solution \(y_{\bar c}\) to the whole interval \((0,1)\). Of course, \(y_{\bar c}(0^+)=0\) and \(y_{\bar c}(\xi)>0\) in \((0,1)\). 

If \(y_{\bar c}(1^-)=0\), then the proof is concluded. Otherwise, if \(y_{\bar c}(1^-)>0\), let us consider a decreasing sequence \(\{\zeta_n\}\) converging to 0, with \(\zeta_1<y_{\overline{c}}(1)\), and the (unique) solution \(\upsilon_n\) of the equation in \eqref{odeprob2} satisfying \(\upsilon_n(1)=\zeta_n\). Similarly to what we observed above, each solution \(\upsilon_n\) is defined in \((0,1]\). By the uniqueness of the solution of the equation \eqref{odeprob2} passing  through a given point in \((0,1)\times (0,+\infty)\), we deduce that \(0<\upsilon_{n+1}(\xi)< \upsilon_n(\xi)<y_{\bar c}(\xi)\) for every \(\xi\in (0,1)\). 
Moreover, by Lemma \ref{l:lowboun} we infer that 
the sequence is also equi-continuous in any compact \(I\subset (0,1)\). So, the sequence is locally uniformly convergent to a function \(\upsilon_{\bar c}\), which is positive in \((0,1)\) by Lemma \ref{l:lowboun}, and is a solution of problem \eqref{odeprob2}.
\end{proof}

\begin{rmk}
\label{finalrmk}
It is worth pointing out that \eqref{eq:c*cond1}--\eqref{eq:c*cond2} imply the estimates \eqref{threshold}, since
$$ \inf_{\xi\in(0,1)} \fint_0^\xi (c^*g(\tau)-f(\tau)) \dtau \geq c^*\inf_{\xi\in(0,1)} \fint_0^\xi g(\tau) \dtau - \sup_{\xi\in(0,1)} \fint_0^\xi f(\tau) \dtau. $$

%an argument analogous to the one at the beginning of \cite[Section 3.1]{DJKZ} leads to

%$$c^*>\frac{\int_0^1 f(\tau)\dtau}{\int_0^1 %g(\tau)\dtau},$$
%which reduces to $c^*>\int_0^1 f(\tau)\dtau$ when %$g\equiv 1$. 
This inequality, joint to what we observed in Remark \ref{r:AC}, shows that our estimates on $c^*$ are finer than the ones in \cite[Theorem 2.1]{DJKZ}. \\
We also highlight that a dual version of Theorem \ref{t:mainthm2} holds when $g(0)<0$: it suffices to change the sign to both the terms \(c\) and \(g\).
\end{rmk}

\section*{Acknowledgments}

\noindent
The authors are member of the {\em Gruppo Nazionale per l'Analisi Matematica, la Probabilit\`a e le loro Applicazioni}
(GNAMPA) of the {\em Istituto Nazionale di Alta Matematica} (INdAM). The first author acknowledges the support of the INdAM-GNAMPA Project 2024 ``Regolarità ed esistenza per operatori anisotropi'' (E5324001950001). \\
This study was funded by the research project of MIUR (Italian Ministry of Education, University and Research) PRIN 2022 {\it Nonlinear differential problems with applications to real phenomena} (Grant No. 2022ZXZTN2).

\end{document}